\documentclass[12pt,reqno]{amsart}

\usepackage{graphicx,amsmath,amsthm,amssymb,fullpage}
\usepackage{epsf, subfigure, verbatim}
\usepackage{color}
\usepackage{srcltx}

\usepackage{amsthm}
\newcommand{\burl}[1]{\textcolor{blue}{\url{#1}}}

\numberwithin{equation}{section}

\newtheorem{thm}{Theorem}[section]

\theoremstyle{plain}

\newtheorem{lemma}[thm]{Lemma}
\newtheorem{proposition}[thm]{Proposition}
\newtheorem{theorem}[thm]{Theorem}
\newtheorem{conjecture}[thm]{Conjecture}
\newtheorem{hypothesis}[thm]{Hypothesis}

\newtheorem{remark}[thm]{Remark}

\usepackage[latin1]{inputenc}
\usepackage[T1]{fontenc}
\usepackage{float}

\usepackage{mathrsfs}
\usepackage{hyperref}
\usepackage{url}

\usepackage{mathrsfs}

\newcommand{\f}{\widehat{\eta}}

\newcommand\be{\begin{equation}}
\newcommand\ee{\end{equation}}
\newcommand\bea{\begin{eqnarray}}
\newcommand\eea{\end{eqnarray}}
\newcommand\bi{\begin{itemize}}
\newcommand\ei{\end{itemize}}
\newcommand\ben{\begin{enumerate}}
\newcommand\een{\end{enumerate}}
\newcommand\bc{\begin{center}}
\newcommand\ec{\end{center}}
\newcommand\ba{\begin{array}}
\newcommand\ea{\end{array}}

\newcommand{\foh}{\frac{1}{2}}  

\newtheorem{rek}[thm]{Remark}

\newcommand{\threecase}[7]{#1 \begin{cases} #2 & \text{{\rm #3}}\\ #4
&\text{{\rm #5}}\\ #6 & \texttt{{\rm #7}} \end{cases}   }

\newcommand{\usp}{\text{USp}}
\newcommand{\soe}{\text{SO(even)}}
\newcommand{\soo}{\text{SO(odd)}}
\newcommand{\so}{\text{O}}

\title{Surpassing the Ratios Conjecture in the 1-level density of Dirichlet $L$-functions}

\author{Daniel Fiorilli}
\email{\textcolor{blue}{\href{fiorilli@umich.edu}{fiorilli@umich.edu}}}
\address{Department of Mathematics, University of Michigan, 530 Church Street, Ann Arbor MI 48109 USA}

\author{Steven J. Miller}
\email{\textcolor{blue}{\href{mailto:sjm1@williams.edu, Steven.Miller.MC.96@aya.yale.edu}{sjm1@williams.edu, Steven.Miller.MC.96@aya.yale.edu}}} \address{Department of Mathematics and Statistics, Williams College, Williamstown, MA 01267}

\subjclass[2010]{11M26, 11M50, 11N13 (primary), 11N56, 15B52 (secondary)}

\keywords{Dirichlet $L$-functions, low-lying zeros, primes in progressions, random matrix theory, ratios conjecture}

\date{\today}

\begin{document}

\thanks{The authors thank Andrew Granville, Chris Hughes, Jeffrey Lagarias, Ze\'ev Rudnick and Peter Sarnak for many enlightening conversations, and the referee for many helpful suggestions. The first named author was supported by an NSERC doctoral, and later on postdoctoral fellowship, as well as NSF grant DMS-0635607, and pursued this work at the Universit\'e de Montr\'eal, the Institute for Advanced Study and the University of Michigan. The second named author was partially supported by NSF grants DMS-0600848, DMS-0970067 and DMS-1265673.}

\begin{abstract}
We study the $1$-level density of low-lying zeros of Dirichlet $L$-functions in the family of all characters modulo $q$, with $Q/2 < q\leq Q$. For test functions whose Fourier transform is supported in $(-3/2, 3/2)$, we calculate this quantity beyond the square-root cancellation expansion arising from the $L$-function Ratios Conjecture of Conrey, Farmer and Zirnbauer. We discover the existence of a new lower-order term which is not predicted by this powerful conjecture. This is the first family where the 1-level density is determined well enough to see a term which is not predicted by the Ratios Conjecture, and proves that the exponent of the error term $Q^{-\frac 12 +\epsilon}$ in the Ratios Conjecture is best possible. We also give more precise results when the support of the Fourier Transform of the test function is restricted to the interval $[-1,1]$. Finally we show how natural conjectures on the distribution of primes in arithmetic progressions allow one to extend the support. The most powerful conjecture is Montgomery's, which implies that the Ratios Conjecture's prediction holds for any finite support up to an error $Q^{-\frac 12 +\epsilon}$.



\end{abstract}

\maketitle

\tableofcontents


\section{Introduction}\label{sec:introduction}

In this paper we study the 1-level density of Dirichlet $L$-functions with modulus $q$. The goal is to compute this statistic for large support and small error terms, providing a test of the predictions of the lower order and error terms in the $L$-function Ratios Conjecture. In this introduction we assume the reader is familiar with low-lying zeros of families of $L$-functions and the Ratios Conjecture, and briefly describe our results. For completeness we provide a brief review of the subject in \S\ref{sec:backgroundprevious}, and state our results in full in \S\ref{sec:unconditionalresults} to \S\ref{sec:resultsbeyondGRH}.

We let $\eta\in  L^1(\mathbb R)$ be an even real function such that $\f$ is $C^2$ and has compact support. Denoting by $\rho_{\chi}=\frac 12+i\gamma_{\chi}$ the non-trivial zeros of $L(s,\chi)$ (i.e., $0<\Re (\rho_{\chi})<1$) and choosing $Q$ a scaling parameter close to $q$,  the 1-level density is\footnote{Since $\widehat \eta$ is $C^2$, we have that $\eta(\xi)\ll \xi^{-2}$ for large $\xi$, hence the sum over the zeros is absolutely convergent. While most of the literature uses $\phi$ as the test function, since we will use Euler's totient function extensively we use $\eta$.}
\begin{equation}\label{main object of interest} D_{1;q}(\eta) \ := \ \frac 1{\phi(q)}\sum_{\chi \bmod q} \sum_{\gamma_{\chi}} \eta\left(\gamma_{\chi} \frac{\log Q}{2\pi} \right);
\end{equation} \emph{throughout this paper a sum over $\chi \bmod q$ always means a sum over all characters, including the principal character.} If we assume GRH then the $\gamma_{\chi}$ are real. As $\eta(y) = \widehat{(\hat\eta)}(y)$ is defined for complex values of $y$, it makes sense to consider \eqref{main object of interest} for complex $\gamma_{\chi}$, in case GRH is false (in other words, GRH is only needed to interpret the 1-level density as a spacing statistic arising from an ordered sequence of real numbers, allowing for a spectral interpretation). We also study the average of \eqref{main object of interest} over the moduli $Q/2<q\leq Q$, which is easier to understand in general: \be\label{eq:onelevelD1QhalvesQ} D_{1;Q/2,Q}(\eta) \ := \ \frac1{Q/2} \sum_{Q/2 < q \le Q} D_{1;q}(\eta). \ee

The powerful Ratios Conjecture of Conrey, Farmer and Zirnbauer \cite{CFZ1,CFZ2} yields a formula for $D_{1;Q/2,Q}(\eta)$ which is believed to hold up to an error $O_{\epsilon}(Q^{-\frac 12+\epsilon})$. While there have been several papers \cite{CS1,CS2,DHP,GJMMNPP,HMM,Mil3,Mil4,MilMo} showing agreement between various statistics involving $L$-functions and the Ratios Conjecture's predictions, evidence for this precise exponent in the error term is limited; the reason this exponent was chosen is the ``philosophy of square-root cancelation''. While some of the families studied have 1-level densities that agree beyond square-root cancelation, it is always for small support (${\rm supp}(\widehat{\eta}) \subset (-1,1)$). Further, in no family studied were non-zero lower order terms beyond square-root cancelation isolated in the $1$-level density.

The motivation of this paper was to resolve these issues. As the ratios conjecture is used in a variety of problems, it is important to test its predictions in the greatest possible window. Our key findings are the following.

\begin{itemize}
\item We uncover new, non-zero lower-order terms in the $1$-level density for our families of Dirichlet characters. These terms are beyond what the Ratios Conjecture can predict, and suggest the possibility that a refinement may be possible and needed.

\item We show (unconditionally) that the natural limit of accuracy of the $L$-function Ratios Conjecture is $\Omega(Q^{-\frac 12+o(1)})$. Thus the error term cannot be improved for a general family of $L$-functions, though of course its veracity for all families is still open.
\end{itemize}




The existence of lower-order terms beyond the Ratios Conjecture's prediction in statistics of $L$-functions is not without precedent. Indeed such terms have been isolated in the second moment of $|L(\tfrac 12,\chi)|$ by Heath-Brown \cite{HB}, and for a more general shifted sum by Conrey \cite{C}.

Before stating our main result, we give the Ratios Conjecture's prediction. This prediction is done for a slightly different family in \cite{GJMMNPP}, but it is trivial to convert from their formulation to the one below (we discuss the conversion in \S\ref{sec:unconditionalresults}).

\begin{conjecture}[Ratios Conjecture]\label{conj:ratiosconj}
The 1-level density $D_{1;q}(\eta)$ (from \eqref{main object of interest} with scaling parameter $Q=q$) equals
\begin{equation}\label{eq:ratiosconjpred} \widehat{\eta}(0) \left(  1-\frac { \log(8\pi e^{\gamma})}{\log q}-\frac{\sum_{p\mid q}\frac{\log p}{p-1}}{\log q}\right) +\int_0^{\infty}\frac{\widehat{\eta}(0)-\widehat{\eta}(t)}{q^{t/2}-q^{-t/2}} dt + O_{\epsilon}\left(q^{-\frac{1}2+\epsilon}\right). \end{equation} The 1-level density $D_{1;Q/2,Q}(\eta)$ (from rescaling\footnote{\label{footnote:rescaling} To rescale we multiply \eqref{eq:ratiosconjpred} by $\log q/\log Q$, replace $q^{t/2}-q^{-t/2}$ with $Q^{t/2}-Q^{-t/2}$ and average over $Q/2 < q \leq Q$. The term $\log q$ averages to $\log Q +\log 2 -1+O(\log Q/Q)$, explaining the ``additional'' term $(\log 2-1)/\log Q$. Moreover the average of $\sum_{p\mid q} \frac{\log p}{p-1}$ over this range is easily shown to be $\sum_{p} \frac 1{p(p-1)} +O(\log Q/Q)$.}  \eqref{eq:ratiosconjpred}) equals \begin{equation} \widehat{\eta}(0) \left( 1 -\frac{\log(4\pi e^{\gamma})+1}{\log Q} - \frac{\sum_p \frac{\log p}{p(p-1)} }{\log Q}\right) + \int_0^{\infty}\frac{\widehat{\eta}(0) - \widehat{\eta}(t)}{Q^{t/2}-Q^{-t/2}} dt + O_{\epsilon}\left(Q^{-\frac{1}2+\epsilon}\right). \end{equation}
\end{conjecture}



Surprisingly, our techniques are capable of not only verifying this prediction, but we are able to determine the 1-level density beyond what even the Ratios Conjecture predicts. In Theorem \ref{first thm} we obtain a new (arithmetical) term of order $Q^{-\frac 12}/\log Q$, which is not predicted by the Ratios Conjecture.

\begin{theorem}\label{first thm} Assume GRH.
If the Fourier Transform of the test function $\eta$ is supported in $(-\frac 32, \frac 32)$, then $D_{1;Q/2,Q}(\eta)$ equals
\begin{equation}
\widehat{\eta}(0) \left( 1 -\frac{\log(4\pi e^{\gamma})+1}{\log Q} - \frac{\sum_p \frac{\log p}{p(p-1)} }{\log Q}\right)+\int_0^{\infty}\frac{\widehat{\eta}(0)-\widehat{\eta}(t)}{Q^{t/2}-Q^{-t/2}} dt + \frac{Q^{-1/2}}{\log Q}S_{\eta}(Q),
 \label{equation 3/2 2}
\end{equation}
where
\be S_{\eta}(Q)\ = \ C_1 \widehat{\eta}(1) + C_2 \frac{\widehat{\eta}'(1)}{\log Q} +  O\left(  \Big(\frac {\log\log Q}{\log Q} \Big)^2\right),\ee
with
\bea C_1 & \ :=\ & (2-\sqrt 2) \zeta\left( \frac 12\right) \prod_p \left( 1+\frac 1{(p-1)p^{1/2}}\right) \nonumber\\
 C_2 & \ :=\ &  C_1  \left(\frac{\sqrt 2+4}3 -\left( \frac{\zeta'}{\zeta} \left(\frac 12\right) -\sum_p \frac{\log p}{(p-1)p^{1/2}+1}\right) \right).\eea

\end{theorem}
We can give a more precise formula for the term $S_{\eta}(Q)$: see Remark \ref{rmk:more precise first thm}. While Theorem \ref{first thm} is conditional on GRH, in Theorem \ref{thm unconditional} we prove a more precise and unconditional result for test functions $\eta$ whose Fourier transform has support contained in $[-1,1]$.

The first two terms in \eqref{equation 3/2 2} agree with the Ratios Conjecture's Prediction. As for the term $Q^{-\frac 12}S_{\eta}(Q)/\log Q$,
its presence confirms that the error term $Q^{-\frac 12+o(1)}$ in the Ratios Conjecture is best possible, and suggests more generally that the $1$-level density of a family ought to contain a (possibly oscillating) arithmetical term of order $Q^{-\frac 12+o(1)}$, a statement which should be tested in other families. Interestingly this new term contains the factors $\hat{\eta}(1)$ and $\f'(1)$, and is zero when $\widehat \eta$ is supported in $(-1,1)$. In this case we give a more precise estimate for the $1$-level density in Theorem \ref{thm unconditional}, in which a lower-order term of order $Q^{\sigma/2 -1 +o(1)}$ appears, where $\sigma= \sup(\text{supp } \f)$. This term is a genuine lower-order term, and shows that for such test functions the Ratios Conjecture's prediction is not best possible. We thus show that a transition happens when $\sigma$ is near $1$. Indeed looking at the difference between the $1$-level density and the Ratios Conjecture's prediction, that is defining \be E_Q(\eta):= D_{1;Q/2,Q}(\eta)- \widehat{\eta}(0) \left( 1 -\frac{\log(4\pi e^{\gamma})+1}{\log Q} - \frac{\sum_p \frac{\log p}{p(p-1)} }{\log  Q}\right) - \int_0^{\infty}\frac{\widehat{\eta}(0) - \widehat{\eta}(t)}{Q^{t/2}-Q^{-t/2}} dt, \ee
our results imply that\footnote{For $\sigma>1$, this holds for test functions $\eta$ for which either $\f(1)\neq 0$ or $\f'(1)\neq 0$ (see Theorem \ref{first thm}); see Theorem \ref{thm unconditional} if $\sigma \le 1$. If $\f(u)$ vanishes in a small interval around $u=1$, then Theorem \ref{thm 3/2} gives the correct answer.} $E_Q(\eta)=Q^{-\mu(\sigma)+o(1)}$, where
\be \mu(\sigma) = \begin{cases}
\frac{\sigma}2 -1 &\text{ if } \sigma \leq 1 \\
-\frac 12 & \text{ if } 1\leq \sigma < \frac 32.
\end{cases}\ee
We conjecture that $\mu(\sigma)$ should equal $-1/2$ for all $\sigma \geq 1$, and that our new lower-order term $Q^{-\frac 12}S_{\eta}(Q)/\log Q $ should persist in this extended range.

\begin{conjecture}
Theorem \ref{first thm} holds for test functions $\eta$ whose Fourier transform has arbitrarily large finite support $\sigma$.
\label{conjecture all support}
\end{conjecture}

In Figure \ref{figure mu(sigma)}, the solid curve represents our results (Theorems \ref{first thm} and \ref{thm unconditional}), and the dashed line represents Conjecture \ref{conjecture all support}; note the resemblance between this graph and the one appearing in Montgomery's pair correlation conjecture \cite{Mon2}. We prove in Theorem \ref{strongest thm} that Montgomery's Conjecture on primes in arithmetic progressions implies that $\mu(\sigma)\leq -1/2$ for all $\sigma \geq 1$.

\begin{figure}
\begin{center}
\includegraphics[scale=.6]{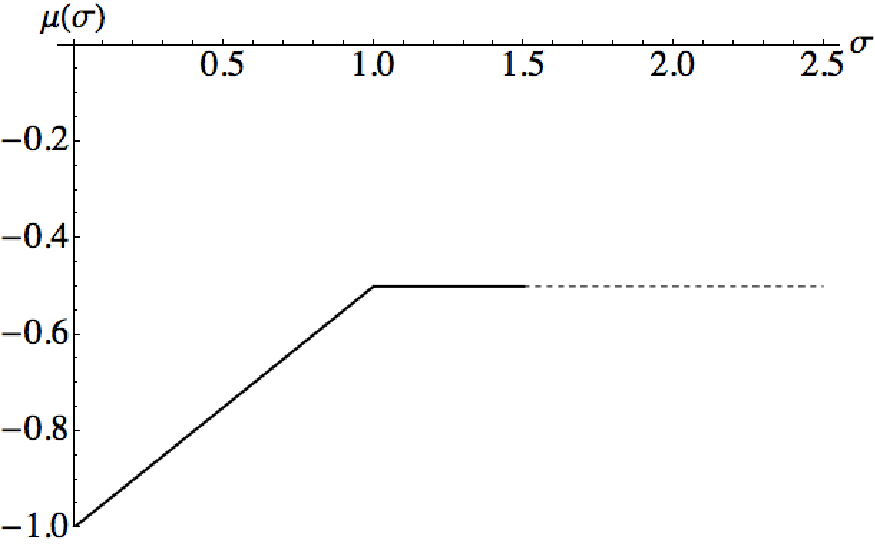}
\caption{\label{figure mu(sigma)} The graph of $\mu(\sigma)$.}
\end{center}
\end{figure}

%
%

We believe that this phenomenon is universal and should also happen in different families, in the sense that we believe that the Ratios Conjecture's prediction should be best possible for $\sigma \geq 1$, and should not be for $\sigma<1$. For example, in \cite{Mil4} it is shown that if the Fourier transform of the involved test function is supported in $(-1,1)$, then the Ratios Conjecture's prediction is not best possible and one can improve the remainder term; however, in this region of limited support there are no new, non-zero lower order terms unpredicted by the Ratios Conjecture.
These results confirm the exceptional nature of the transition point $\sigma =1$, as is the case in Montgomery's Pair Correlation Conjecture \cite{Mon2}. Indeed if this last conjecture were known to hold beyond the point $\alpha =1$, then this would imply the non-existence of Landau-Siegel zeros \cite{CI}.

Our plan of attack is to use the explicit formula to turn the $1$-level density into an average of the various terms appearing in this formula. The bulk of the work is devoted to carefully estimating the contribution of the prime sum, which when summing over $\chi \bmod q$ becomes a sum over primes in the residue class $1 \bmod q$, averaged over $q\sim Q$. Accordingly, the proof of Theorem \ref{first thm} is based on ideas used in the recent results of the first named author \cite{fiorilli}, which improve on results of Fouvry \cite{fouvry}, Bombieri, Friedlander and Iwaniec \cite{BFI}, Friedlander and Granville \cite{FG} and Friedlander, Granville, Hildebrandt and Maier \cite{FGHM}. Theorem 1.1 of \cite{fiorilli} cannot be applied directly here, since this estimate is only valid when looking at primes up to $x$ modulo $q$ with $q\sim Q$, where $Q$ is not too close to $x$. Additional estimates are needed, including a careful analysis of the range $x^{1-\epsilon} < Q \leq x$, which required a combination of divisor switching techniques and precise estimates on the mean value of smoothed sums of the reciprocal of Euler's totient function. Additionally, in our analysis of the $1$-level density after using the explicit formula and executing the sum over the family we obtain a sum over primes in the arithmetic progressions $1\bmod q$; this is one of the cases in which one obtains an asymptotic in Theorem 1.1 of \cite{fiorilli}, which explains the occurrence of the lower-order term $Q^{-\frac 12} S_{\eta}(Q)/\log Q$ in Theorem \ref{first thm}.

%
%

The paper is organized as follows. In \S\ref{sec:backgroundprevious} we review previous results on low-lying zeros in families of $L$-functions and describe the motivation for the Ratios Conjecture. See for example \cite{GJMMNPP, Mil4} for a detailed description of how to apply the Ratios Conjecture to predict the 1-level density. We describe our unconditional results in \S\ref{sec:unconditionalresults}, and then improve our results in \S\ref{sec:resultsunderGRH} by assuming GRH. In previous families there often is a natural barrier, and extending the support is related to standard conjectures (for example, in \cite{ILS} the authors show  how cancelation in exponential sums involving square-roots of primes leads to larger support for families of cuspidal newforms). A similar phenomenon surfaces here, where in \S\ref{sec:resultsbeyondGRH} we show that increasing the support beyond $(-2,2)$ is related to conjectures on the distribution of primes in residue classes. We analyze the increase in support provided by various conjectures. These range from a conjecture on the variance of primes in the residue classes, which allow us to reach $(-4, 4)$, to Montgomery's conjecture for a fixed residue, which gives us any finite support. The next sections contain the details of the proof; we state the explicit formula and prove some needed sums in \S\ref{sec:expformulaestimates}, and then prove our theorems in the remaining sections.


\section{Background and New Results}

\subsection{Background and Previous Results}\label{sec:backgroundprevious}


Assuming GRH, the non-trivial zeros of any nice $L$-function lie on the critical line, and therefore it is possible to investigate statistics of its normalized zeros. These zeros are fundamental in many problems, ranging from the distribution of primes in congruence classes to the class number \cite{CI,Go,GZ,RubSa}. Numerical and theoretical evidence \cite{Hej,Mon2,Od1,Od2,RS} support a universality in behavior of zeros of an individual automorphic $L$-function high above the central point, specifically that they are well-modeled by ensembles of random matrices (see \cite{FM,Ha} for histories of the emergence of random matrix theory in number theory). The story is different for the low-lying zeros, the zeros near the central point. A convenient way to study these zeros is via the 1-level density, which we now describe. Let $\eta\in L^1(\mathbb R)$ be an even real function whose Fourier transform
\begin{equation}
  \label{eq:9}
 \hat{\eta}(y)\ =\ \int_{-\infty}^\infty \eta(x) e^{-2\pi ixy}dx
\end{equation}
is $C^2$ and has compact support. Let $\mathcal{F}_N$ be a (finite) family of
$L$-functions satisfying GRH.\footnote{\label{footnote:spectral}We often do not need GRH for the analysis, but only to interpret the results. If the GRH is true, the zeros lie on the critical line and can be ordered, which suggests the possibility of a spectral interpretation.} The $1$-level density associated to
$\mathcal{F}_N$ is defined by
\begin{equation}
\label{eq:7} D_{1;\mathcal{F}_N}(\eta)\ =\ \frac1{|\mathcal{F}_N|}
\sum_{g\in\mathcal{F}_N} \sum_{j} \eta\left(\frac{\log
c_g }{2\pi}\gamma_g^{(j)}\right),
\end{equation}
where $\foh + i\gamma_g^{(j)}$ runs through the non-trivial zeros
of $L(s,g)$. Here $c_g$ is the ``analytic conductor'' of $g$, and
gives the natural scale for the low zeros. As $\eta$ decays,
only low-lying zeros (i.e., zeros within a distance
$1/\log c_g$ of the central point $s=1/2$) contribute
significantly. Thus the $1$-level density can help identify the
symmetry type of the family. To evaluate \eqref{eq:7}, one applies the explicit formula,
converting sums over zeros to sums over primes.

Based in part on the function-field analysis where $G(\mathcal{F})$ is the monodromy group associated to the family
$\mathcal{F}$, Katz and Sarnak conjectured that for each reasonable irreducible family of $L$-functions there is an associated symmetry group $G(\mathcal{F})$ (one of the following five: unitary~$U$, symplectic~$\usp$, orthogonal~$\so$, $\soe$, $\soo$), and that the distribution of critical zeros near $1/2$ mirrors the distribution of eigenvalues near~$1$. The five groups have distinguishable $1$-level densities. To date, for suitably restricted test functions the statistics of zeros of many natural families of $L$-functions have been shown to agree with statistics of eigenvalues of matrices from the classical compact groups, including Dirichlet $L$-functions, elliptic curves, cuspidal newforms, Maass forms, number field $L$-functions, and symmetric powers of ${\rm GL}_2$ automorphic representations \cite{AM,AAILMZ,DM1,FI,Gao,Gu,HM,HR,ILS,KaSa1,KaSa2,Mil1,MilPe,RR,Ro,Rub,ShTe,Ya,Yo2}, to name a few, as well as non-simple families formed by Rankin-Selberg convolution \cite{DM2}.

In addition to predicting the main term (see for example \cite{Con,KaSa1,KaSa2,KeSn1,KeSn2,KeSn3}), techniques from random matrix theory have led to models that capture the lower order terms in their full arithmetic glory for many families of $L$-functions (see for example the moment conjectures of \cite{CFKRS} or the hybrid model in \cite{GHK}). Since the main terms agree with either unitary, symplectic or orthogonal symmetry, it is only in the lower order terms that we can break this universality and see the arithmetic of the family  enter. These are therefore natural and important objects to study, and can be isolated in many families \cite{HKS,Mil2,Yo1}. We thus require a theory that is capable of making detailed predictions. Recently the $L$-function Ratios Conjecture \cite{CFZ1,CFZ2} has had great success in determining lower order terms. Though a proof of the Ratios Conjecture for arbitrary support is well beyond the reach of current methods, it is an indispensable tool in current investigations as it allows us to easily write down the predicted answer to a remarkable level of precision, which we try to prove in as great a generality as possible.


To study the 1-level density, it suffices to obtain good estimates for \be R_{\mathcal{F}_N}(\alpha,\gamma)\ :=\ \frac1{|\mathcal{F}_N|} \sum_{g \in \mathcal{F}_N} \frac{L(1/2 + \alpha, g)}{L(1/2 + \gamma, g)}. \ee (In the current paper, the parameter $Q$ plays the role of $|\mathcal{F}_N|$.) Asymptotic formulas for $R_{\mathcal{F}_N}(\alpha,\gamma)$ have been conjectured for a variety of families $\mathcal{F}_N$ (see \cite{CFZ1,CS1,CS2,GJMMNPP,HMM,Mil3,Mil4,MilMo}) and are believed to hold up to errors of size $|\mathcal{F}_N|^{-1/2+\epsilon}$ for any $\epsilon > 0$. The evidence for the correctness of this error term is limited to test functions with small support (frequently  significantly less than $(-1,1)$), though in such regimes many of the above papers verify this prediction. Many of the steps in the Ratios Conjecture's recipe lead to the addition or omission of terms as large as those being considered, and thus there was uncertainty as to whether or not the resulting predictions should be accurate to square-root cancelation. The results of the current paper can be seen as a confirmation that this is the right error term for the final predicted answer, at least in this family.  Further, the novelty in our results resides in the fact that we are able to go beyond square-root cancelation and we find a smaller term which is unpredicted by the Ratios Conjecture (see Theorem \ref{first thm}). For a precise explanation on how to derive the Ratios Conjecture's prediction in our family, we refer the reader to \cite{GJMMNPP}, and also recommend \cite{CS1} for an accessible overview of the Ratios Conjecture.

\subsection{Unconditional Results}\label{sec:unconditionalresults}
We now describe our unconditional results. We remind the reader that $\eta$ is a real even function such that $\widehat \eta$ is $C^2$ and has compact support.
\begin{theorem}
\label{thm unconditional}
Suppose that the Fourier transform of the test function $\eta$ is supported on the interval $[-1,1]$, so $\sigma = \sup(\text{\emph{supp} } \f)\leq 1$. There exists an absolute positive constant $c$ (coming from the Prime Number Theorem) such that the 1-level density $D_{1;q}(\eta)$ (from \eqref{main object of interest} with scaling parameter $Q=q$) equals
\bea\label{eq:d1qhatfexp} & & \f(0) \left(  1-\frac { \log(8\pi e^{\gamma})}{\log q}-\frac{\sum_{p\mid q}\frac{\log p}{p-1}}{\log q}\right)
+\int_0^{\infty}\frac{\f(0)-\f(t)}{q^{t/2}-q^{-t/2}} dt \nonumber\\
& &\ \ \ \ \  -\frac 2{\phi(q)} \int_0^{1} q^{u/2}\left(\frac {\f(u)}2 -\frac {\f'(u)}{\log q}\right) du -\ \frac 2{\log q} \sum_{\substack{p^{\nu}\parallel q \\ p^e \equiv 1 \bmod q/p^{\nu} \\ e, \nu \geq 1}}\frac{\log p}{\phi(p^{\nu})p^{e/2}}\f\left( \frac{\log p^{e}}{\log q}\right)\nonumber\\ & & \ \ \ \ \  +\ O\left(\frac{q^{\frac{\sigma}2-1}}{e^{c \sqrt{\sigma\log q}}}\right).\eea
\end{theorem}

\begin{remark}\label{rek:beatratios}
The average over $Q/2<q \leq Q$ of the fourth term in \eqref{eq:d1qhatfexp} can be shown to be $O(Q^{-1})$, and is therefore negligible when considering $D_{1;Q/2,Q}(\eta)$ (see \eqref{eq:average of T_2 is small}). However, the term involving the second integral in \eqref{eq:d1qhatfexp} is of size $q^{\sigma/2-1-o(1)} $, and thus constitutes a genuine lower-order term, smaller than the error term in \eqref{eq:ratiosconjpred} predicted using the Ratios Conjecture.
\end{remark}

Theorems \ref{first thm} and \ref{thm unconditional} should be compared to the main result of Goes, Jackson, Miller, Montague, Ninsuwan, Peckner and Pham  \cite{GJMMNPP}, where they show one can extend the support of $\f$ to $[-2,2]$ and still get the main term, as well as the lower order terms down to a power savings. They only consider $q$ prime, and thus the sum over primes $p$ dividing $q$ below in Theorem \ref{thm millergoesal}   is absorbed by their error term. We briefly discuss how one can easily extend their results to the case of general $q$. First note that $L(s,\chi)$ and $L(s,\chi^\ast)$ have the same zeros in the critical strip if $\chi^\ast$ is the primitive character of conductor $q^\ast$ inducing the non-principal character $\chi$ of conductor $q$. We now have $\log q^\ast$, which can be converted to a sum over primes $p$ dividing $q$ by the same arguments as in the proof of Proposition \ref{proposition explicit formula}. The rest of the expansion follows from expanding the digamma function $\Gamma'/\Gamma$ in the integral in Theorem 1.3 of \cite{GJMMNPP} and then standard algebra (along the lines of the computations in \S\ref{sec:expformulaestimates}). We use Lemma 12.14 of \cite{montgomery}, which in our notation says that for $a, b > 0$ we have \be \int_{-\infty}^\infty \frac{\Gamma'(a \pm i b \tau)}{\Gamma(a \pm i b \tau)} \eta(t) dt \ = \ \frac{\Gamma'(a)}{\Gamma(a)} \widehat{\eta}(0) + \frac{2\pi}{b} \int_{0}^\infty \frac{\exp(-2\pi a x / b)}{1 - \exp(-2\pi x/b)} \left(\widehat{\eta}(0) - \widehat{\eta}(\mp x)\right)dx,\ee and the identity \be \frac{\Gamma'(1/4)}{\Gamma(1/4)} + \frac{\Gamma'(3/4)}{\Gamma(3/4)} \ = \ -2\gamma - 6 \log 2, \ee with $\gamma$ the Euler-Mascheroni constant. We then extend to $q \in (Q/2, Q]$ by rescaling the zeros by $\log Q$ and not $\log q$ and summing over the family; recall the technical issues involved in the rescaling are discussed in Footnote \ref{footnote:rescaling}.

\begin{theorem}[Goes, Jackson, Miller, Montague, Ninsuwan, Peckner, Pham \cite{GJMMNPP}]
\label{thm millergoesal}
If $1<\sigma\leq 2$, then the 1-level density $D_{1;q}(\eta)$ (from \eqref{main object of interest} with scaling parameter $Q=q$) equals
\begin{equation}\label{equation support [-2,2]}\f(0) \left(  1-\frac { \log(8\pi e^{\gamma})}{\log q}-\frac{\sum_{p\mid q}\frac{\log p}{p-1}}{\log q}\right)
+\int_0^{\infty}\frac{\f(0)-\f(t)}{q^{t/2}-q^{-t/2}} dt
    +O\left(\frac{\log\log q}{\log q}q^{\frac{\sigma}2-1}\right),
 \end{equation} and this agrees with the Ratios Conjecture.
\end{theorem}

\begin{remark}
Goes et al. \cite{GJMMNPP} actually proved \eqref{equation support [-2,2]} for any $\sigma\leq 2$, with the additional error term $O(q^{-1/2+\epsilon})$. We prefered not to include the case $\sigma\leq 1$, as Theorem \ref{thm unconditional} is more precise in this range.
\end{remark}



\subsection{Results under GRH}\label{sec:resultsunderGRH}

We first mention a more precise version of Theorem \ref{first thm}.
\begin{remark}
\label{rmk:more precise first thm}
If in addition to the hypotheses of Theorem \ref{first thm} we assume that the Fourier transform of the test function $\eta$ is $K+1$ times continuously differentiable, then we can give a more precise expression for the term $S_{\eta}(Q)$ appearing in \eqref{equation 3/2 2}: \be S_{\eta}(Q) \ =\ \sum_{i=0}^K \frac{a_i(\eta)}{(\log Q)^i} + O_{\epsilon,K}\left( \frac {1}{(\log Q)^{K+1-\epsilon}}\right), \ee where the $a_i(\eta)$ are constants depending (linearly) on the Taylor coefficients of $\f(t)$ at $t=1$. In fact, $S_{\eta}(Q)$ is a truncated linear functional, which composed with the Fourier Transform operator is supported on $\{1\}$ (in the sense of distributions).
\end{remark}

Our next result is an extension of Theorem \ref{first thm}, in the case where $\hat{\eta}(u)$ vanishes in a small interval to the right of $u=1$.
\begin{theorem}\label{thm 3/2} Assume GRH.

\begin{enumerate}

\item If $\hat{\eta}$ is supported in $(-\frac 32, -1-\kappa] \cup [-1,1] \cup [1+\kappa, \frac 32)$ for some $\kappa>0$, then for any $\epsilon >0$ the average 1-level density $D_{1;Q/2,Q}(\eta)$ equals
\begin{eqnarray}\label{equation 3/2 3}
& & \hat{\eta}(0) \left( 1-\frac{1+\log(4\pi e^{\gamma})}{\log Q} -\frac{\sum_p \frac{\log p}{p(p-1)}}{\log Q}\right)+\int_0^{\infty}\frac{\f(0)-\f(t)}{Q^{t/2}-Q^{-t/2}} dt
    \nonumber\\
& & \ \ \ \ \ - \ \frac{4\log 2}Q \frac{\zeta(2)\zeta(3)}{\zeta(6)} \int_0^{1} Q^{u/2}\left(\frac {\hat{\eta}(u)}2 -\frac {\hat{\eta}'(u)}{\log Q}\right) du
\nonumber\\
& & \ \ \ \ \  - \ \int_{1+\kappa}^{4/3}(  (u-1)\log Q+C_6  ) Q^{-u/2} \left(\frac {\hat{\eta}(u)}2 -\frac {\hat{\eta}'(u)}{\log Q}\right) du \nonumber \\
& & \ \ \ \ \ + \ O_{\epsilon}(Q^{-\frac 12-\kappa+\epsilon}+Q^{-\frac 23}\log Q+Q^{\sigma-2}\log Q),
\end{eqnarray}
with $C_6:= \log (\pi/2) +1+ \gamma + \sum_p \frac{\log p}{p(p-1)}$.\\
Note that for $\sigma \geq \frac 43$, unless $\hat{\eta}(x)$ has some mass near $x=\lambda$ for some $1<\lambda<4-2\sigma$, the fourth term in \eqref{equation 3/2 3} goes in the error term (and hence \eqref{equation 3/2 3} reduces to \eqref{equation 3/2 1}). However, if $1<\sigma<\frac 43$, it is always a genuine lower-order term of size $Q^{-\sigma/2+o(1)}$.

\item If $f$ is supported in $(-2,-a] \cup [-1,1] \cup [a,2)$ for some $1\leq a<2$ (if $a=1$, then we have the full interval $(-2,2)$), then we have that $D_{1;Q/2,Q}(\eta)$ equals
\begin{eqnarray}
& & \hat{\eta}(0) \left( 1-\frac{1+\log(4\pi e^{\gamma})}{\log Q} -\frac{\sum_p \frac{\log p}{p(p-1)}}{\log Q}\right)+\int_0^{\infty}\frac{\f(0)-\f(t)}{Q^{t/2}-Q^{-t/2}} dt
    \nonumber\\ & & \ \ \ \ \ \ \ - \  \frac{4\log 2}Q \frac{\zeta(2)\zeta(3)}{\zeta(6)} \int_0^{1} Q^{u/2}\left(\frac {\f(u)}2 -\frac {\f'(u)}{\log Q}\right) du
+ O(Q^{-\frac a2}+Q^{\sigma-2}\log Q).
\label{equation 3/2 1}
\end{eqnarray}
Unless $a>1$ and $\sigma<\frac 32$, the third term of \eqref{equation 3/2 1} goes in the error term.\\

\end{enumerate}

\end{theorem}

\subsection{Results beyond GRH}\label{sec:resultsbeyondGRH}

As the GRH is insufficient to compute the 1-level density for test functions supported beyond $[-2, 2]$, we explore the consequences of other standard conjectures in number theory involving the distribution of primes among residue classes. Before stating these conjectures, we first set the notation. Let
\bea \psi(x)\ :=\  \sum_{n \le x} \Lambda(n), \hspace{1cm}  \psi(x,q,a)\ :=\ \sum_{\substack{n \le x \\ n \equiv a \bmod q}} \Lambda(n), \eea
\be  E(x,q,a) \ :=\  \psi(x,q,a) -
\frac{\psi(x)}{\phi(q)}.\ee If we assume GRH, we have that

\begin{equation}
\label{eq:GRHbounds} \psi(x) = x + O(x^{\frac 12}(\log x)^2), \hspace{1.5cm} E(x,q,a)  = O(x^{\frac 12} (\log x)^2).
\end{equation}

Our first result uses GRH and the following de-averaging hypothesis, which depends on a parameter $\delta \in [0,1]$.

\begin{hypothesis}\label{de-averaging hypothesis}
We have
\be\label{eq:deavehypotheta} \sum_{Q/2<q\leq Q}\left|\psi(x;q,1) - \frac{\psi(x)}{\phi(q)} \right|^2\ \ll\ Q^{\delta-1} \sum_{Q/2<q\leq Q} \sum_{\substack{1\leq a \leq q : \\ (a,q)=1}} \left| \psi(x;q,a) - \frac{\psi(x)}{\phi(q)}\right|^2.\ee
\end{hypothesis}

This hypothesis is trivially true for $\delta = 1$, and while it is unlikely to be true for $\delta = 0$, it is reasonable to expect it to hold for any $\delta>0$. What we need is some control over biases of primes congruent to $1 \bmod q$. For the residue class $a \bmod q$,
$\left|\psi(x;q,a) - \frac{\psi(x)}{\phi(q)}\right|^2$ is the variance; the above conjecture can be
interpreted as bounding $\left|\psi(x;q,1) - \frac{\psi(x)}{\phi(q)}\right|^2$ in terms of the average
variance.\footnote{Note that we only need this de-averaging hypothesis for the special residue class $a=1$. }

\stepcounter{X}

Under these hypotheses, we show how to extend the support to a wider but still limited range.

\begin{theorem}
Assume GRH and Hypothesis \ref{de-averaging hypothesis} for some $\delta \in (0,1)$. The average 1-level density $D_{1;Q/2,Q}(\f)$ equals
\begin{multline} \hat{\eta}(0) \left( 1-\frac{1+\log(4\pi e^{\gamma})}{\log Q} -\frac{\sum_p \frac{\log p}{p(p-1)}}{\log Q}\right)+\int_0^{\infty}\frac{\f(0)-\f(t)}{Q^{t/2}-Q^{-t/2}} dt  \\
+O(Q^{\frac{\delta-1}2} (\log Q)^{\frac 32}+Q^{\frac{\sigma+2\delta}4-1} (\log Q)^{\frac 13}),
\label{equation thm 4-2eta}
\end{multline}
which is asymptotic to $\f(0)$ provided the support of $\f$ is contained in $(-4+2\delta,4-2\delta)$.
\label{thm 4-2eta}
\end{theorem}

The proof of Theorem \ref{thm 4-2eta} is given in \S\ref{sec:proofthm42eta}. It uses a result of Goldston and Vaughan \cite{GV}, which is an improvement of results of Barban, Davenport, Halberstam, Hooley, Montgomery and others.

\begin{remark}  In Theorem \ref{thm 4-2eta} we study the weighted 1-level density
\begin{equation} D_{1;Q/2,Q}(\eta) \ := \    \sum_{Q/2<q\leq Q}\frac 1{\phi(q)}\sum_{\chi \bmod q} \sum_{\gamma_{\chi}} \eta\left(
\gamma_{\chi} \frac{\log Q}{2\pi} \right), \end{equation} which is technically easier to study than the unweighted version
\begin{equation} D_{1;Q/2,Q}^{\rm unweighted}(\eta) \ := \ \frac 1{\frac 9{\pi^2}(Q/2)^2} \sum_{Q/2<q\leq Q}\sum_{\chi \bmod q} \sum_{\gamma_{\chi}} \eta\left(
\gamma_{\chi} \frac{\log Q}{2\pi} \right).
\end{equation}
This is similar to many other families of $L$-functions, such as cuspidal newforms \cite{ILS,MilMo} and Maass forms \cite{AAILMZ,AM}, where the introduction of weights (arising from the Petersson and Kuznetsov trace formulas) facilitates evaluating the arithmetical terms.
\end{remark}

Finally, we show how we can determine the 1-level density for arbitrary finite support, under a hypothesis of Montgomery \cite{Mon1}.

\begin{hypothesis}[Montgomery]
\label{montgomery original}
For any $a,q$ such that $(a,q)=1$ and $q\leq x$, we have
\be \psi(x;q,a)-\frac {\psi(x)}{\phi(q)}\ \ll_{\epsilon}\ x^{\epsilon} \left( \frac xq\right)^{1/2}.  \ee
\end{hypothesis}



It is by gaining some savings in $q$ in the error $E(x,q,a)$ that we can increase the support for families of Dirichlet $L$-functions. The following weaker version of Montgomery's Conjecture, which depends on a parameter $\theta \in (0, 1/2]$, also suffices to increase the support beyond $[-2, 2]$.

\begin{hypothesis}
\label{montgomery weaker}
For any $a,q$ such that $(a,q)=1$ and $q\leq x$, we have
\be \psi(x;q,1)-\frac {\psi(x)}{\phi(q)}\ \ll_{\epsilon}\ \frac{x^{\frac 12+\epsilon}}{q^{\theta}} .  \ee
\end{hypothesis}



\begin{hypothesis}
\label{montgomery weakest}
Fix $\epsilon>0$. We have for $x^{\epsilon}\leq q\leq \sqrt x$ that
\begin{equation} \sum_{\substack{n\leq x \\ n\equiv 1\bmod q}} \Lambda(n)\left( 1-\frac nx\right)-\frac{1}{\phi(q)}\sum_{n\leq x}\Lambda(n)\left( 1-\frac nx\right) \ = \ o( x^{1/2}).
\label{small o equation}
\end{equation}
\end{hypothesis}

Note that this is a weighted version of $\psi(x;q,1)-\frac{\psi(x)}{\phi(q)}$; that is, we added the weight $\left( 1-\frac nx\right)$. The reason for this is that it makes the count smoother, and this makes it easier to analyze in general since the Mellin transform of $g(y):=1-y$ in the interval $[0,1]$ is decaying faster in vertical strips than that of $g(y)\equiv 1$.

Amongst the last three hypotheses, Hypothesis \ref{montgomery weakest} is the weakest, but it is still sufficient to derive the asymptotic in the $1$-level density for test functions with arbitrary large support.

\begin{theorem}\label{strongest thm} For $\eta$ whose Fourier Transform has arbitrarily large (but finite) support, we have the following:

\begin{enumerate}

\item  If we assume Hypothesis \ref{montgomery weakest}, then the 1-level density $D_{1;q}(\eta)$ equals $\f(0)+o(1)$, agreeing with the scaling limit of unitary matrices.

\item If we assume Hypothesis \ref{montgomery weaker} for some $0<\theta\leq \frac 12$, then $D_{1;q}(\eta)$ equals
\begin{equation}\f(0) \left(  1-\frac { \log(8\pi e^{\gamma})}{\log q}-\frac{\sum_{p\mid q}\frac{\log p}{p-1}}{\log q}\right)
+\int_0^{\infty}\frac{\f(0)-\f(t)}{q^{t/2}-q^{-t/2}} dt +O_{\epsilon}(q^{-\theta+\epsilon}).
\label{equation consequence hyp theta}
 \end{equation}
\end{enumerate}
\end{theorem}

\begin{remark}
Under GRH, the left hand side of \eqref{small o equation} is $O(x^{1/2}\log q)$. Therefore, if we win by more than a logarithm over GRH, then we have the expected asymptotic for the 1-level density for $\f$ of arbitrarily large finite support.

Interestingly, if we assume Montgomery's Conjecture (Hypothesis \ref{montgomery original}), then we can take $\theta=1/2$ in \eqref{equation consequence hyp theta}, and doing so we end up precisely with the Ratios Conjecture's prediction (see \eqref{eq:ratiosconjpred}).
\end{remark}

We derive the explicit formula for the families of Dirichlet characters in \S\ref{sec:expformulaestimates}, as well as some useful estimates for some of the resulting sums. We give the unconditional results in \S\ref{section terms}, Theorems \ref{thm unconditional} and \ref{thm millergoesal}. The proofs of Theorems \ref{first thm} and \ref{thm 3/2} are conditional on GRH, and use results of \cite{FG} and \cite{fiorilli}; we give them in \S\ref{sec:proofthm32}. We conclude with an analysis of the consequences of the hypotheses on the distribution of primes in residue classes, using the de-averaging hypothesis to prove Theorem \ref{thm 4-2eta} in \S\ref{sec:proofthm42eta} and Montgomery's hypothesis to prove Theorem \ref{strongest thm} in \S\ref{sec:montstrong}.

\section{The Explicit Formula and Needed Sums}\label{sec:expformulaestimates}

Our starting point for investigating the behavior of low-lying zeros is the explicit formula, which relates sums over zeros to sums over primes. We follow the derivation in \cite{montgomery} (see also \cite{ILS,RS}, and \cite{Da,IK} for all needed results about Dirichlet $L$-functions). We first derive the expansion for Dirichlet characters with fixed conductor $q$, and then extend to $q \in (Q/2, Q]$. We conclude with some technical estimates that will be of use in proving Theorem \ref{first thm}. Here and throughout, we will set $f:=\f$. Note that $\eta$ is real and even, and thus so is the case for $f$, and moreover we have $\widehat f = \eta$.

\subsection{The Explicit Formula for fixed $q$}

\begin{proposition}[Explicit Formula for the Family of Dirichlet Characters Modulo $q$]
\label{proposition explicit formula}
Let $f$ be an even, twice differentiable test function with compact support. Denote the non-trivial zeros of $L(s,\chi)$ by $\rho_\chi=1/2+i\gamma_\chi$. Then the 1-level density $D_{1,q}(\widehat{f})$ equals
\begin{align}
\frac 1{\phi(q)}&\sum_{\chi \bmod q} \sum_{\gamma_{\chi}} \widehat{f}\left(
\gamma_{\chi} \frac{\log Q}{2\pi} \right)\ =\  \frac{f(0)}{\log Q} \left(  \log q -\log(8\pi e^{\gamma})-\sum_{p\mid q}\frac{\log p}{p-1}\right)\notag \\
&+\int_0^{\infty}\frac{f(0)-f(t)}{Q^{t/2}-Q^{-t/2}} dt -\frac 2{\log Q} \sum_{\substack{p^{\nu}\parallel q \\ p^e \equiv 1 \bmod q/p^{\nu} \\ e, \nu \geq 1}}\frac{\log p}{\phi(p^{\nu})p^{e/2}}f\left( \frac{\log p^{e}}{\log Q}\right)\notag
\\ &-\frac{2}{\log Q} \left(\sum_{n\equiv 1 \bmod q}-\frac 1{\phi(q)}\sum_{n} \right)\frac{\Lambda(n)}{n^{1/2}} f\left(
\frac{\log n}{\log Q}\right)+O\left(\frac{1}{\phi(q)}\right).
\label{explicit formula}
\end{align}
\end{proposition}

\begin{proof}
We start with Weil's explicit formula for $L(s,\chi)$, with $\chi\bmod q$ a non-principal character (we add the contribution from the principal character later). We can replace $L(s,\chi)$ by $L(s,\chi^*)$ (where $\chi^*$ is the primitive character of conductor $q^*$ inducing $\chi$), since these have the same non-trivial zeros.   Taking $F(x):=\frac{2\pi}{\log Q}f\left(\frac {2\pi x}{\log Q}\right)$ in Theorem 12.13 of \cite{montgomery} (whose conditions are satisfied by our restrictions on $f$), we find $\Phi(s)=\widehat{f}\left( \frac{\log Q}{2\pi} \frac{(s-\frac 12)}i\right)$, and
\begin{eqnarray}\label{eq:firstrhosum} & & \sum_{\rho_{\chi}} \widehat{f}\left( \frac{\log Q}{2\pi} \gamma_{\chi}\right)\ =\ \frac {f(0)}{\log Q}\left( \log(q^*/\pi)+\frac{\Gamma'}{\Gamma}\left(\frac 14+\frac {a(\chi)}2\right)\right) \nonumber\\ & &
-\frac 2{\log Q} \sum_{n=1}^{\infty} \frac{\Lambda(n) \Re(\chi^*(n))}{n^{1/2}}f\left( \frac{\log n}{\log Q}\right)
+\frac{4\pi}{\log Q}\int_0^{\infty} \frac{e^{-(1+2a(\chi))\pi x}}{1-e^{-4\pi x}} \left(f(0)-f\left(\frac{2\pi x}{\log Q}\right)\right)dx, \nonumber\\
\end{eqnarray} where $a(\chi) = 0$ for the half of the characters with $\chi(-1) = 1$ and $1$ for the half with $\chi(-1) = -1$. Making the substitution $t=\frac{2\pi x}{\log Q}$ in the integral and summing over $\chi\neq \chi_0$, we find
\begin{eqnarray}\label{eq:secondrhosum}& & \sum_{\chi \neq \chi_0} \sum_{\gamma_{\chi}} \widehat{f}\left(
\gamma_{\chi} \frac{\log Q}{2\pi} \right) \  = \  \frac {f(0)}{\log Q}\left( \sum_{\chi\neq \chi_0} \log(q^*/\pi)+\frac{\phi(q)}{2}\frac{\Gamma'}{\Gamma}\left(\frac 34\right)+\frac{\phi(q)}{2}\frac{\Gamma'}{\Gamma}\left(\frac 14\right)\right) \nonumber\\
& & \ \ \ \ \ + \  \phi(q)\int_0^{\infty}\frac{Q^{-3t/2}+Q^{-t/2}}{1-Q^{-2t}}(f(0)-f(t))dt \nonumber\\
& & \ \ \ \ \ - \ \frac 2{\log Q} \left(\phi(q)\sum_{n\equiv 1 \bmod q} - \sum_n \right) \frac{\Lambda(n)}{n^{1/2}}f\left( \frac{\log n}{\log Q}\right) \nonumber\\
& & \ \ \ \ \  - \ \frac 2{\log Q} \sum_{\chi \neq \chi_0} \sum_{n} \frac{\Lambda(n) \Re(\chi^*(n)-\chi(n))}{n^{1/2}}f\left( \frac{\log n}{\log Q}\right) +O\left(1\right).
\end{eqnarray} To get \eqref{eq:secondrhosum} from \eqref{eq:firstrhosum} we added zero by writing $\chi^*(n)$ as $\left(\chi^*(n) - \chi(n)\right) + \chi(n)$. Summing $\chi(n)$ over all $\chi \bmod q$ gives $\phi(q)$ if $n \equiv 1 \bmod q$ and $0$ otherwise; as our sum omits the principal character, the sum of $\chi(n)$ over the non-principal characters yields the sum on the third line above. We also replaced $(\phi(q)-1)/2$ by $\phi(q)/2$ in the first term, hence the $O(1)$.

We use Proposition 3.3 of \cite{FiMa} for the first term (which involves the sum over the conductor of the inducing character). We then use the duplication formula of the digamma function $\psi(z) = \Gamma'(z)/\Gamma(z)$ to simplify the next two terms, namely $\psi(1/4) + \psi(3/4)$. As $\psi(1/2) = -\gamma-2\ln 2$ (equation 6.3.3 of \cite{AS}) and $\psi(2z) = \frac12\psi(z) + \frac12\psi(z+\frac12)+\ln 2$ (equation 6.3.8 of \cite{AS}), setting $z=1/4$ yields $\psi(1/4)+\psi(3/4) = -2\gamma -6\ln 2$. We keep the next two terms as they are, and then apply Proposition 3.4 of \cite{FiMa} (with $r=1$) for the last term, obtaining that it equals
\be -\frac{2}{\log Q} \sum_n \frac{\Lambda(n)}{n^{1/2}} f\left(\frac{\log n}{\log Q}\right) {\rm Re}\left(    \sum_{\chi\neq \chi_0}   \left(\chi^*(n)-\chi(n)\right)\right). \ee
Writing $n=p^e$, this term is zero unless $p\mid q$. If $p\mid q$, then it is zero unless $p^e\equiv 1 \bmod
q/p^\nu$, where $\nu\geq 1$ is the largest $\nu$ such that $p^\nu\mid q$. Therefore this term equals
\be -\frac{2}{\log Q} \sum_p \sum_{\substack{e, \nu \ge 1 \\ p^\nu \parallel q, p^e \equiv 1 \bmod q/p^\nu}} \frac{\Lambda(p^e)}{
\phi(p^\nu)p^{e/2}} f\left(\frac{\log p^e}{\log Q}\right).\ee Combining the above and some elementary algebra yields
\begin{eqnarray} & & \frac 1{\phi(q)}\sum_{\chi \neq \chi_0} \sum_{\gamma_{\chi}} \widehat{f}\left(
\gamma_{\chi} \frac{\log Q}{2\pi} \right) \ = \ \frac{f(0)}{\log Q} \left(  \log q -\log(8\pi e^{\gamma})-\sum_{p\mid q}\frac{\log p}{p-1}\right) \nonumber\\
& & \ \ \ \ \ + \ \int_0^{\infty}\frac{f(0)-f(t)}{Q^{t/2}-Q^{-t/2}} dt  - \ \frac 2{\log Q} \left(\sum_{n\equiv 1 \bmod q} - \frac 1{\phi(q)}\sum_n \right) \frac{\Lambda(n)}{n^{1/2}}f\left( \frac{\log n}{\log Q}\right) \nonumber\\
& & \ \ \ \ \ - \ \frac 2{\log Q} \sum_{\substack{p^{\nu}\parallel q \\ p^e \equiv 1 \bmod q/p^{\nu} \\ e, \nu\geq 1}}\frac{\log p}{\phi(p^{\nu})p^{e/2}}f\left( \frac{\log p^{e}}{\log Q}\right)   +O\left(\frac 1{\phi(q)}\right).
\label{almost there functional equation}
\end{eqnarray}
Finally, since the non-trivial zeros of $L(s,\chi_0)$ coincide with those of $\zeta(s)$, the difference between the left hand side of \eqref{explicit formula} and that of \eqref{almost there functional equation} is
\be \frac 1{\phi(q)} \sum_{\gamma_\zeta} \widehat{f}\left(
\gamma_\zeta \frac{\log Q}{2\pi} \right)\ \ll\ \frac 1{\phi(q)}\ee
(since $f$ is twice continuously differentiable, $\widehat{f}(y) \ll 1/y^2$),
completing the proof.\footnote{While the explicit formula for $\zeta(s)$ has a term arising from its pole at $s=1$, that term does not matter here as it is insignificant upon division by the family's size.}
\end{proof}

\subsection{The Averaged Explicit Formula for $q \in (Q/2, Q]$}

We now average the explicit formula for $D_{1;q}(\widehat{f})$ (Proposition \ref{proposition explicit formula}) over $q \in (Q/2,Q]$. We concentrate on deriving useful expansions, which we then analyze in later sections when we determine the allowable support.

\begin{proposition}[Explicit Formula for the Averaged Family of Dirichlet Characters Modulo $q$]\label{proposition averaged explicit formula} The averaged 1-level density, $D_{1;Q/2,Q}(\widehat{f})$, equals \bea\label{eq:D1q2qave} & & D_{1;Q/2,Q}(\widehat{f}) \ = \
\frac1{Q/2} \sum_{Q/2 < q \le Q} D_{1;q}(\widehat{f}) \nonumber\\ && = \ \frac{f(0)}{\log Q} \left( \log Q -1-\gamma-\log(4\pi) -\sum_p \frac{\log p}{p(p-1)}\right)  + \ \int_0^{\infty}\frac{f(0)-f(t)}{Q^{t/2}-Q^{-t/2}} dt \nonumber\\ & & \ \ \ + \ \frac{2}{Q/2} \sum_{Q/2 < q \le Q} \int_0^{\infty} \left(\frac {f(u)}2 -\frac {f'(u)}{\log Q}\right)\frac{\psi(Q^u;q,1)-\frac{\psi(Q^u)}{\phi(q)}}{Q^{u/2}}\ du + O\left(\frac1{Q}\right). \eea
Setting \be\label{eq:psi2def} \psi_2(x;q,a) \ :=\ \sum_{\substack{n\leq x\\n\equiv a \bmod q}} \Lambda(n)\left(1-\frac nx\right), \hspace{1cm} \psi_2(x) \ :=\ \sum_{\substack{n\leq x}} \Lambda(n)\left(1-\frac nx\right), \ee the last integral in \eqref{eq:D1q2qave} may be replaced with
\be -2 \int_0^{\infty} \left(\frac {3f(u)}4 -\frac {2f'(u)}{\log Q} +\frac {f''(u)}{(\log Q)^2}\right)\frac{\psi_2(Q^u;q,1)-\frac{\psi_2(Q^u)}{\phi(q)}}{Q^{u/2}}\ du.
\ee
\end{proposition}

\begin{proof}

The main term in the expansion of $D_{1;q}(\widehat{f})$ from Proposition \ref{proposition explicit formula} is \be T_1(q) \ := \ \frac{f(0)}{\log Q} \left(  \log q -\log(8\pi e^{\gamma})-\sum_{p\mid q}\frac{\log p}{p-1}\right).\ee
Using the anti-derivative of $\log x$ is $x \log x - x$, one easily finds its average over $Q/2<q\leq Q$ is
\be \frac 1{Q/2} \sum_{Q<q\leq 2Q} T_1(q)\ =\ \frac{f(0)}{\log Q} \left( \log Q -1-\gamma-\log(4\pi) -\sum_p \frac{\log p}{p(p-1)}\right)+O\left( \frac 1{Q} \right).\ee

We now turn to the lower-order term
\be T_2(q)\ := \ -\frac 2{\log Q} \sum_{\substack{p^{\nu}\parallel q \\ p^e \equiv 1 \bmod q/p^{\nu} \\ e, \nu \geq 1}}\frac{\log p}{\phi(p^{\nu})p^{e/2}}f\left( \frac{\log p^{e}}{\log Q}\right).\ee Before determining its average behavior, we note that its size can vary greatly with $q$. It is very small for prime $q$ (so $\nu=1$ and $p=q$ in the sum), since
\be T_2(q)\  \ll\ \frac 1{\log Q} \sum_{e\geq 1} \frac{\log q}{\phi(q)q^{e/2}}\ \ll\ \frac 1{(q-1)(q^{\frac 12}-1)}; \ee
however, it can be as large as $\frac C{\sqrt q \log Q}$ for other values of $q$ (such as $q=2(2^e-1)$). This is more or less as large as it can get, since for general $q$ we have
\be T_2(q) \ \ll\ \frac 1{\log Q} \sum_{\substack{p^{\nu}\parallel q \\ e, \nu \geq 1 \\ p^e \leq Q^{\sigma}}}\frac{\log p}{\phi(p^{\nu})(q/p^{\nu})^{1/2}}\ \ll\ \frac{(\log q)^{\frac 12}}{q^{\frac 12}\log\log q}. \label{bound on T2}\ee
On average, however, $T_2(q)$ is very small:
\begin{eqnarray} \frac 1{Q/2} \sum_{Q/2<q\leq Q} T_2(q) & \ \ll\ & \frac 1Q\sum_{Q/2<q\leq Q} \sum_{\substack{p^{\nu}\parallel q \\ p^e \equiv 1 \bmod q/p^{\nu} \\ e, \nu \geq 1}}\frac{\log p}{p^{\nu+e/2}}\ \ll\  \frac 1Q \sum_{\substack{p^{\nu} \\ \nu, e \geq 1}} \frac{\log p}{p^{\nu+e/2}} \sum_{\substack{q\leq Q \\ p^{\nu}\mid q \\ \frac{q}{p^{\nu}}\mid p^e-1}} 1 \nonumber\\
& \ll& \frac 1Q \sum_{\substack{p^{\nu} \\ \nu, e\geq 1}} \frac{\log p}{p^{\nu+e/2}} \tau(p^e-1)\ \ll_{\epsilon}\ \frac 1Q \sum_{\substack{p^{\nu} \\ \nu, e\geq 1}} \frac{\log p}{p^{\nu+(1-\epsilon) \frac e2}} \nonumber\\
&\ll& \frac 1Q \sum_{p} \frac{\log p}{p^{\frac 32-\frac{\epsilon}2}} \ \ll\ \frac 1Q.
\label{eq:average of T_2 is small}
 \end{eqnarray}

While we will not re-write the next lower order term, it is instructive to determine its size. Set
\be T_3(q) \ :=\ \int_0^{\infty}\frac{f(0)-f(t)}{Q^{t/2}-Q^{-t/2}} dt. \ee Letting $t = 2\pi x / \log Q$, we find
\be\label{eq:T3qwithexp} T_3(q) \ = \ \frac{2\pi}{\log Q}\int_0^{\infty} \frac{f(0)-f\left(\frac{2\pi x}{\log Q}\right)}{2\sinh(\pi x)}dx.\ee Since $f$ is twice differentiable with compact support, $|f(0)-f(x)|\ll |x|$, thus \be T_3(q) \ \ll \ \frac{2\pi}{\log Q} \int_0^\infty \frac{x}{2\sinh(\pi x)}\ dx \ = \ \frac{\pi}{4\log Q}.\ee
As \be \int_0^\infty \frac{x^k dx}{\sinh(\pi x)} \ = \ \frac{2^{k+1}-1}{2^k \pi^{k+1}} \Gamma(k+1)\zeta(k+1),\ee
if $f$ has a Taylor series of order $K+1$ we have \be T_3(q) \ = \ \sum_{k=1}^K \frac{(2^{k+1}-1) \zeta(k+1)f^{(k)}(0)}{\log^{k+1} Q} + O\left(\frac{1}{\log^{K+1} Q}\right).\ee
%
If the Taylor coefficients of $f$ decay very fast, we can even make our bounds uniform and get an error term smaller than a negative power of $Q$.


The remaining term from Proposition \ref{proposition explicit formula} is the most important, and controls the allowable support. The arithmetic lives here, as this term involves primes in arithmetic progressions. It is
\begin{eqnarray}\label{penultimate prime sum} T_4(q) & \ := \ &  -\frac{2}{\log Q} \left(\sum_{n\equiv 1 \bmod q}-\frac 1{\phi(q)}\sum_{n} \right)\frac{\Lambda(n)}{n^{1/2}} f\left(
\frac{\log n}{\log Q}\right) \nonumber\\
&=&  -\frac{2}{\log Q} \int_1^{\infty} t^{-\frac 12} f\left( \frac{\log t}{\log Q}\right) d\left(\psi(t;q,1)-\frac{\psi(t)}{\phi(q)}\right) \notag \\
&=& \frac{2}{\log Q} \int_1^{\infty} \frac{\frac 12 f\left( \frac{\log t}{\log Q}\right)-\frac 1{\log Q}f'\left( \frac{\log t}{\log Q}\right) }{t^{\frac 32}}\left(\psi(t;q,1)-\frac{\psi(t)}{\phi(q)}\right) dt. \end{eqnarray}
The claim in the proposition follows by changing variables by setting $t = Q^u$; specifically, the final integral is
\begin{eqnarray}\label{prime sum term} T_4(q)
&\ = \ &2 \int_0^{\infty} \left(\frac {f(u)}2 -\frac {f'(u)}{\log Q}\right)\frac{\psi(Q^u;q,1)-\frac{\psi(Q^u)}{\phi(q)}}{Q^{u/2}}\ du.
\end{eqnarray}

We give an alternative expansion for the final integral. This expansion involves a smoothed sum of $\Lambda(n)$, which will be technically easier to analyze when we turn to determining the allowable support under Montgomery's hypothesis (Theorem \ref{strongest thm}(1)). Recall \be \psi_2(x;q,a)\ :=\ \sum_{\substack{n\leq x\\n\equiv a \bmod q}} \Lambda(n)\left(1-\frac nx\right), \hspace{1cm} \psi_2(x) \ :=\ \sum_{\substack{n\leq x}} \Lambda(n)\left(1-\frac nx\right). \ee We integrate by parts in \eqref{penultimate prime sum}. Since \begin{eqnarray}
& &  \int_1^{x} \left(\psi(t;q,1)-\frac{\psi(t)}{\phi(q)}\right) dt \ = \  \int_1^{x} \left(\sum_{\substack{n\leq t \\ n\equiv 1 \bmod q}} \Lambda(n) - \frac 1{\phi(q)} \sum_{n\leq t} \Lambda(n)\right) dt \nonumber\\
& & \ \ \ \ \ = \ \sum_{\substack{n\leq x \\ n\equiv 1 \bmod q}}\Lambda(n) \int_n^x dt - \frac 1{\phi(q)} \sum_{n\leq x} \Lambda(n)\int_n^x dt \nonumber\\
& & \ \ \ \ \ = \  x\left( \sum_{\substack{n\leq x \\ n\equiv 1 \bmod q}}\Lambda(n) \left(1-\frac nx \right) - \frac 1{\phi(q)} \sum_{n\leq x} \Lambda(n)\left(1-\frac nx \right) \right),
\end{eqnarray}
we find \begin{equation}T_4(q) \ = \ -2 \int_0^{\infty} \left(\frac {3f(u)}4 -\frac {2f'(u)}{\log Q} +\frac {f''(u)}{(\log Q)^2}\right)\frac{\psi_2(Q^u;q,1)-\frac{\psi_2(Q^u)}{\phi(q)}}{Q^{u/2}}\ du,
\end{equation}
completing the proof.
\end{proof}

\begin{rek} It will be convenient later that in the averaged case $\psi$ and $\psi_2$ are both evaluated at $(Q^u;q,1)$ and not $(q^u;q,1)$; this is because we are rescaling all $L$-function zeros by the same quantity (a global rescaling instead of a local rescaling).  \end{rek}

\subsection{Technical Estimates}

In the proof of Theorem \ref{thm 3/2}, we need the following estimation of a weighted sum of the reciprocal of the totient function.
\begin{lemma}
\label{sum reciprocal totient lemma}
Let $\phi$ be Euler's totient function. We have
\begin{equation}  \sum_{r\leq R} \frac 1{\phi(r)}\left( R^{1/2} + \frac r{R^{1/2}} -2r^{1/2} \right)=D_1R^{1/2} \log R + D_2 R^{1/2} + D_3 + O\left(\frac{\log R}{R^{1/2}}\right),
\label{sum reciprocal totient equation}
\end{equation}
where
$$ D_1 \ :=\  \frac{\zeta(2)\zeta(3)}{\zeta(6)}, \hspace{2cm}  D_2\ :=\ D_1 \left( \gamma-3-\sum_p \frac{\log p}{p^2-p+1} \right), $$
\be D_3\ :=\ -2 \zeta\left(\frac 12\right)\prod_p \left(1+\frac 1{(p-1)p^{1/2}} \right).\ee

More generally, if $P(u):=\sum_{i=0}^d a_i u^i$ is a polynomial of degree $d$ and of norm \be \Vert P\Vert \ :=\ \max_{i} |a_i|,\ee then
\begin{multline}  \sum_{r\leq R} \frac 1{\phi(r)} \int_{\frac{\log r}{\log R}} ^1 P(u) \left(R^{\frac u2} -\frac r{R^{\frac u2}} \right)du
\ = \ E_1 \log R  \int_{-\infty}^1 R^{\frac u2}uP(u)du \\+ E_2 \int_{-\infty}^1 R^{\frac u2}P(u)du
+ \sum_{j=1}^{d+1} \frac{F_j(P)}{(\log R)^j}  +O_d(R^{-\frac 12}\Vert P\Vert)
\label{sum reciprocal totient equation general}
\end{multline}
where
\begin{equation}
E_1 \ := \  \frac{\zeta(2)\zeta(3)}{\zeta(6)}, \hspace{1cm} E_2  \ := \ E_1 \left( \gamma-1-\sum_p \frac{\log p}{p^2-p+1} \right),
\label{equation definition E_i}
\end{equation}
and the $F_j(P)$ are constants depending on $P$ which can be computed explicitly.
For example,
\bea F_1(P) & = &  -4\zeta\left( \frac 12\right) \prod_p \left( 1+\frac 1{(p-1)p^{1/2}}\right) \sum_{i=0}^d (-1)^i P^{(i)}(1) \nonumber\\
 F_2(P) &=&  -4\zeta\left( \frac 12\right) \prod_p \left( 1+\frac 1{(p-1)p^{1/2}}\right) \left( \frac{\zeta'}{\zeta}\left( \frac 12 \right)-\sum_p \frac{\log p}{(p-1)p^{1/2} +1} \right) \sum_{i=1}^d (-1)^i P^{(i)}(1). \nonumber\\ \eea
Finally,
\begin{eqnarray} & &  \sum_{r\leq R} \frac 1{\phi(r)} \int_{\frac{\log (r/2)}{\log (R/2)}}^1 P(u) \left((R/2)^{\frac u2} -\frac r{2(R/2)^{\frac u2}} \right)du
\nonumber\\ & &  = \ E_1 \log (R/2)  \int_{-\infty}^1 (R/2)^{\frac u2}uP(u)du
+ (E_2+E_1\log 2) \int_{-\infty}^1 (R/2)^{\frac u2}P(u)du \nonumber\\
& & \ \ \ \ +\ \sum_{j=1}^{d+1} \frac{F_j^{(2)}(P)}{(\log (R/2))^j}  +O_d(R^{-\frac 12}\Vert P\Vert),
\label{sum reciprocal totient equation general 2}
\end{eqnarray}
where the first two constants are given by
\bea F_1^{(2)}(P) & \ :=  \ & \frac{F_1(P)}{\sqrt 2} \nonumber\\
F_2^{(2)}(P) &:= & -2\sqrt 2\zeta\left( \frac 12\right) \prod_p \left( 1+\frac 1{(p-1)p^{1/2}}\right)\nonumber\\
& & \ \ \ \ \  \times\ \left( \frac{\zeta'}{\zeta}\left( \frac 12 \right)-\sum_p \frac{\log p}{(p-1)p^{1/2} +1}+\log 2 \right) \sum_{i=1}^d (-1)^i P^{(i)}(1).
\eea


\end{lemma}
\begin{remark}  It is possible to improve the estimates in \eqref{sum reciprocal totient equation}, \eqref{sum reciprocal totient equation general} and \eqref{sum reciprocal totient equation general 2} to ones with an error term of $O_{\epsilon,d}(R^{-5/4+\epsilon}\Vert P\Vert)$; however, this is not needed for our purposes.
\end{remark}

\begin{proof}
 By Mellin inversion, for $c\geq 2$ the left hand side of \eqref{sum reciprocal totient equation} equals
\begin{eqnarray} \frac 1{2\pi i} \int_{\Re (s)=c} Z(s) \left( \frac{R^{s+\frac 12}}{s}+\frac{R^{s+\frac 12}}{s+1}-2\frac{R^{s+\frac 12}}{s+\frac 12} \right)  ds  \ = \ \frac 1{2\pi i} \int_{\Re (s)=c} Z(s) \frac{R^{s+\frac 12}}{2s(s+\frac 12)(s+1)} ds, \nonumber\\
\label{integral to compute with residues}
\end{eqnarray}
where \be Z(s) \ :=\ \sum_{n\geq 1} \frac 1{n^s \phi(n)}.\ee
Taking Euler products,
\be Z(s) \ = \ \zeta(s+1)\zeta(s+2)Z_2(s),\ee
where
\be Z_2(s) \ :=\ \prod_p \left( 1+\frac 1{p(p-1)}\left(\frac 1{p^{s+1}}-\frac 1{p^{2s+2}} \right)\right),\ee
which converges for $\Re (s)>-\frac 32$. We shift the contour of integration to the left to the line $\Re (s)=-\frac 32+\epsilon$. By a standard residue calculation, we get that \eqref{integral to compute with residues} equals
 \be D_1R^{1/2} \log R + D_2 R^{1/2} + D_3 + D_4\frac{\log R}{R^{1/2}} + \frac{D_5}{R^{1/2}}+\frac 1{2\pi i} \int_{\Re (s)=-\frac 32+\epsilon} Z(s) \frac{R^{s+\frac 12}}{2s(s+\frac 12)(s+1)} ds\ee
for some constants $D_4$ and $D_5$. The proof now follows from standard bounds on the zeta function, which show that this integral is $\ll_{\epsilon} R^{-1+\epsilon}$. See the proof of Lemma 6.9 of \cite{fiorilli} for more details.

We now move to \eqref{sum reciprocal totient equation general}. The Mellin transform in this case is (for $\Re (s)>0$)
\begin{eqnarray} \alpha(s)& \ := \ & \int_0^R r^{s-1}\int_{\frac{\log r}{\log R}} ^1 P(u) \left(R^{\frac u2} -\frac r{R^{\frac u2}} \right) du dr \nonumber\\
&=& \int_{-\infty}^1 P(u) \int_0^{R^u } r^{s-1} \left(R^{\frac u2} -\frac r{R^{\frac u2}} \right) dr du \nonumber\\
&=& \int_{-\infty}^1 P(u) \frac{R^{u(s+\frac 12)}}{s(s+1)}\ du,
\end{eqnarray}
which is now defined for $\Re(s)>-1/2$. To meromorphically extend $\alpha(s)$ to the whole complex plane, we integrate by parts $n$ times:
\begin{equation} \alpha(s)\ =\ \frac{R^{s+\frac 12}}{s(s+1)}\sum_{i=0}^{n}  \frac {(-1)^i P^{(i)}(1)}{(s+\frac 12)^{i+1} (\log R)^{i+1}},
\label{analytic continuation of the mellin transform}
\end{equation}
which is a meromorphic function with poles at the points $s=0,-1/2, -1$. The integral we need to compute is
\be \frac 1{2\pi i} \int_{\Re (s)=1} Z(s) \alpha(s) ds. \ee
We remark that
\be \alpha(-3/2 +\epsilon+it)\ \ll_{\epsilon,d}\ \frac {R^{- 1+\epsilon}}{t^3}\Vert P\Vert,\ee
hence the proof is similar as in the previous case, since by shifting the contour of integration to the left, we have
\be \frac 1{2\pi i} \int_{\Re (s)=1} Z(s) \alpha(s) ds\ =\ A+O_{\epsilon,d}(R^{-1+\epsilon}\Vert P\Vert), \ee
where $A$ is the sum of the residues of $Z(s) \alpha(s)$ for $-3/2+\epsilon \leq \Re (s) \leq 2$. Note that if $\beta(s):=s(s+1)\alpha(s)$, then
\be \beta(0) \ = \ \int_{-\infty}^1 R^{\frac u2}P(u)du, \hspace{2cm} \beta'(0)\ =\ \log R \int_{-\infty}^1 R^{\frac u2}uP(u)du, \ee
so the residue at $s=0$ equals
\be \frac{\zeta(2)\zeta(3)}{\zeta(6)}\beta(0) \left( \frac{\beta'}{\beta}(0)+\gamma-1- \sum_p\frac{\log p}{p^2-p+1}  \right).  \ee
For the pole at $s=-1/2$, we need to use the analytic continuation of $\alpha(s)$ provided in \eqref{analytic continuation of the mellin transform}, which shows that this residue equals
\be \sum_{j=1}^{n+1} \frac{F_j(P)}{(\log R)^j},\ee
where the $F_j(P)$ are constants depending on $P$ which can be computed explicitly. For example,
\bea F_1(P) & =&  -4\zeta\left( \frac 12\right) \prod_p \left( 1+\frac 1{(p-1)p^{1/2}}\right) \sum_{i=0}^d (-1)^i P^{(i)}(1) \nonumber\\  F_2(P) &=&  -4\zeta\left( \frac 12\right) \prod_p \left( 1+\frac 1{(p-1)p^{1/2}}\right) \left( \frac{\zeta'}{\zeta}\left( \frac 12 \right)-\sum_p \frac{\log p}{(p-1)p^{1/2} +1} \right) \sum_{i=1}^d (-1)^i P^{(i)}(1).\nonumber\\ \eea
Moreover, $F_i(P) \ll_d \Vert P\Vert$ for all $i$.

At $s=-1$, we have a double pole with residue
\be R^{-\frac 12} \sum_{j=0}^{n+1} \frac{G_j(P)}{(\log R)^j}, \ee
for some constants $G_j(P)\ll_d \Vert P \Vert$, hence the the proof of \eqref{sum reciprocal totient equation general} is complete.

For the proof of \eqref{sum reciprocal totient equation general 2}, we proceed in the same way, noting that the Mellin transform is
\begin{eqnarray} \alpha_2(s) \ =\  \frac{2^s}{s(s+1)}\int_{-\infty}^1 P(u) (R/2)^{u(s+\frac 12)}du.
\end{eqnarray}

\end{proof}

\section{Unconditional Results (Theorems \ref{thm unconditional} and \ref{thm millergoesal})}\label{section terms}

Using the expansion for the 1-level density $D_{1,q}(\widehat{f})$ and the averaged 1-level density $D_{1;Q/2,Q}(\widehat{f})$ from Propositions \ref{proposition explicit formula} and \ref{proposition averaged explicit formula}, we prove our unconditional results.

\begin{proof}[Proof of Theorem \ref{thm unconditional}]
We start from Proposition \ref{proposition explicit formula}. The only term of \eqref{explicit formula} we need to understand is the last one (the ``prime sum''), which is given by
\be T_4(q) \ := \ 2 \int_0^{1} \left(\frac {f(u)}2 -\frac {f'(u)}{\log q}\right)\frac{\psi(q^u;q,1)-\frac{\psi(q^u)}{\phi(q)}}{q^{u/2}}\ du.\ee
(We used that the support of $f$ is contained in $[-1,1]$ and we made the substitution $t=q^u$.) However, since there are no integers congruent to $1\bmod q$ in the interval $[2,q^u]$ when $u\leq 1$ (this is also true when $q^u$ is replaced by $Q^u$, with $Q/2<q\leq Q$), the $\psi(q^u;q,1)$ term equals zero. By the Prime Number Theorem there is an absolute, computable constant $c>0$ such that
\begin{eqnarray} T_4(q) & \ = \ &  -2 \int_0^{1} \left(\frac {f(u)}2 -\frac {f'(u)}{\log q}\right)\frac{\psi(q^u)}{q^{u/2}\phi(q)}\ du \nonumber\\ &= & -\frac 2{\phi(q)} \int_0^{1} q^{u/2}\left(\frac {f(u)}2 -\frac {f'(u)}{\log q}\right) du
+O\left( \frac 1{\phi(q)} \int_0^{\sigma} \frac{q^{u/2}}{e^{c\sqrt{u\log q}}}\ du\right),
\end{eqnarray}
and the error term is \be \ll\ \frac{q^{\sigma/4}}{\phi(q)}\int_0^{\sigma/2} e^{-c\sqrt{u\log q}}du+\frac{e^{-c\sqrt{\frac{\sigma}2 \log q}}}{\phi(q)}\int_{\sigma/2}^{\sigma} q^{u/2}du\ \ll\ \frac{q^{\sigma/2-1}}{e^{c'\sqrt{\sigma\log q}}}
\ee
for $q$ large enough (in terms of $\sigma$), completing the proof.
\end{proof}

\begin{proof}[Proof of Theorem \ref{thm millergoesal}]
Starting again from Proposition \ref{explicit formula}, we have that
\be -\frac 2{\log Q} \sum_{\substack{p^{\nu}\parallel q \\ p^e \equiv 1 \bmod q/p^{\nu} \\ e, \nu \geq 1}}\frac{\log p}{\phi(p^{\nu})p^{e/2}}f\left( \frac{\log p^{e}}{\log Q}\right)\ \ll\ \frac{(\log q)^{\frac 12}}{q^{\frac 12}\log\log q} \ee
(see \eqref{bound on T2}), hence this goes in the error term and the only term we need to worry about is the last one.

As our support exceeds $[-1, 1]$, the $\psi(q^u;q,1)$ no longer trivially vanishes, and the last term is
\be T_4(q) \ = \  2 \int_0^{2} \left(\frac {f(u)}2 -\frac {f'(u)}{\log q}\right)\frac{\psi(q^u;q,1)-\frac{\psi(q^u)}{\phi(q)}}{q^{u/2}}\ du.\ee In the proof of Theorem \ref{thm unconditional} above we showed that the contribution from the integral where $0\leq u\leq 1$ is $O(q^{-1/2})$.

For any fixed $\epsilon > 0$, trivial bounds for the region $1\leq u \leq 1+\epsilon$ yield a contribution that is
\be \ll\ \int_1^{1+\epsilon} (u\log q) q^{\frac u2-1}du \ \ll\ q^{-\frac 12+\epsilon}.\ee

We use the Brun-Titchmarsh Theorem (see \cite{thelargesieve}) for the region where $1+\epsilon \leq u \leq 2$, which asserts that for $q<x$, \be \pi(x;q,a) \ \leq\ \frac{2 x}{\phi(q)\log(x/q)}.\ee We first bound the contribution from prime powers as follows. First there are at most $2e^{\omega(q)}$ residue classes $b\bmod q$ such that $b^e\equiv 1 \bmod q$, and so using that $\omega(q) \ll \log q/\log\log q$ we compute
\begin{eqnarray} \sum_{e\geq 2} \sum_{\substack{p\leq x^{1/e} \\ p^e \equiv 1 \bmod q}} \log p & \ \ll\ & \sum_{2\leq e \leq \frac 2{\epsilon}} e^{\omega(q)} \max_{b \bmod q}\Bigg(\sum_{\substack{p\leq x^{1/e} \\ p\equiv b \bmod q}}\log p \Bigg)+ \sum_{\frac 2{\epsilon} \leq e\leq 2\log x} \sum_{p\leq x^{1/e}}\log p \nonumber\\ & \ll & \sum_{2\leq e\leq \frac 2{\epsilon}} e^{\omega(q)} \left( 1+\frac{x^{1/e}}q\right) \log x + \sum_{\frac 2{\epsilon} \leq e\leq 2\log x} x^{1/e}
\nonumber\\ &\ll & \left( \frac 2{\epsilon}\right)^{\omega(q)+1}\left(1+\frac{x^{1/2}}q\right) \log x+x^{\epsilon/2}\log x
\nonumber\\ &\ll_{\epsilon} & x^{\epsilon}\left(1+\frac{x^{1/2}}q\right),
\end{eqnarray}
provided $q$ is large enough in terms of $\epsilon$.

Thus, for $1+\epsilon \leq u\leq 2$, we have
\be \psi(q^u;q,1)\ \ll_{\epsilon}\ \frac{q^{u-1} \log (q^u)\log\log q}{(u-1)\log q}+ q^{\epsilon}+q^{\frac u2-1+\epsilon}\ \ll_{\epsilon}\ q^{u-1} \log\log q,\ee
which bounds the integral from $1+\epsilon$ to $\sigma$ by
\be \ll\ \int_{1+\epsilon}^{\sigma} q^{\frac u2-1}\log\log q du\ \ll\ \frac{\log\log q}{\log q}q^{\frac{\sigma}2-1},\ee
completing the proof.
\end{proof}

\section{Results Under GRH (Theorems \ref{first thm} and \ref{thm 3/2})}\label{sec:proofthm32}

In this section we assume GRH (but none of the stronger results about the distribution of primes among residue classes) and prove Theorems \ref{first thm} and \ref{thm 3/2}. The main ingredient in the proofs are the results of \cite{fouvry, BFI, FG, fiorilli}. The following is the needed conditional version.

\begin{theorem}\label{thm fiorilli} Assume GRH. Fix an integer $a\neq 0$ and $\epsilon>0$. We have for $M=M(x)\leq x^{\frac 14}$ that
\be   \sum_{\substack{\frac x{2M}<q\leq \frac x{M} \\ (q,a)=1}} \left( \psi(x;q,a)-\Lambda(a)-\frac {\psi(x)}{\phi(q)}\right)\ =\ \frac{\phi(a)}{a}\frac{x}{2M} \mu_0(a,M) +O_{a,\epsilon} \left(\frac{x}{M^{\frac{3}{2}-\epsilon}}+\sqrt x M (\log x)^2 \right), \ee
where
\be \threecase{\mu_0(a,M) \ := \ }{-\frac 12 \log M-\frac{C_6}2}{if $a=\pm 1$}
{-\frac 12 \log p}{if $a=\pm p^e$}{0}{otherwise,}
\ee
with
\be C_6 \ :=\ \log \pi +1+ \gamma + \sum_p \frac{\log p}{p(p-1)}.\ee
\end{theorem}

\begin{proof}
See Remark 1.5 of \cite{fiorilli}. Note that the restriction $M=o(x^{\frac 14}/\log x)$ is required for the error term to be negligible compared to the main term, but it can be changed to $M\leq x^{\frac 14}$.

\end{proof}

We now proceed to prove Theorems \ref{first thm} and \ref{thm 3/2}. Note that by the averaged 1-level density (Proposition \ref{proposition averaged explicit formula}), the proof is completed by analyzing the average of $T_4(q)$:
\begin{equation} \frac1{Q/2} \sum_{Q/2<q\leq Q}T_4(q)\ = \ 2 \int_0^{\sigma} \left(\frac {f(u)}2 -\frac {f'(u)}{\log Q}\right)\frac1{Q/2}\sum_{Q/2<q\leq Q}\frac{\psi(Q^u;q,1)-\frac{\psi(Q^u)}{\phi(q)}}{Q^{u/2}}\ du.
 \label{integral to understand}
\end{equation}

We break the integral into regions and bound each separately. Going through the proof of Theorem \ref{thm unconditional} and applying GRH, we see that the contribution to the integral from $u \in [0, 1]$ equals
\be- \frac{4\log 2}Q \frac{\zeta(2)\zeta(3)}{\zeta(6)} \int_0^{1} Q^{u/2}\left(\frac {f(u)}2 -\frac {f'(u)}{\log Q}\right) du
+O\left( \frac{\log^2 Q}Q\right).
\label{part of the integral u<1}\ee
We now analyze the three cases of the theorem, corresponding to different support restrictions for our test function.

\begin{proof}[Proof of Theorem \ref{thm 3/2}(2)] To prove \eqref{equation 3/2 1}, we need to understand the part of the integral in \eqref{integral to understand} with $a\leq u \leq 2$. Arguing as in \cite{FG} (see also the proof of Proposition 6.1 of \cite{fiorilli}), we have that for $x^{1/2} \leq Q \leq x$,
\be \sum_{Q/2<q\leq Q} \left(\psi(x;q,1)-\frac{\psi(x)}{\phi(q)} \right)\ \ll\ Q \left(\log (x/Q)+1\right) + \frac{x^{3/2} (\log x)^2}Q.
\label{equation Q can be sqrt}
\ee
The basic idea to obtain this last estimate is to write
$$ \sum_{Q/2<q\leq Q} \psi(x;q,1) = \sum_{\substack{ n\leq x \\ n-1 = qr \\ Q/2<q\leq Q }} \Lambda(n),$$
and to turn this into a sum over $r\leq 2(x-1)/Q$ of the function $\psi(x;r,1)-\psi(rQ/2+1;r,1)$. One then applies GRH and estimates the resulting sum over $r$ using estimates on the summatory function of $1/\phi(r)$. Applying \eqref{equation Q can be sqrt}, the part of the integral in \eqref{integral to understand} with $a\leq u \leq 2$ is
\be \ll\ \int_{a}^{\sigma} \left( Q^{-u/2} (\log (Q^{u-1}) +1)  + Q^{u-2} (\log (Q^u))^2\right)du\ \ll\ Q^{-\frac a2} + Q^{\sigma-2} \log Q. \label{integral between a and sigma}\ee \ \\

\end{proof}

\begin{proof}[Proof of Theorem \ref{thm 3/2}(1)]
We need to study the part of the integral in \eqref{integral to understand} with $1+\kappa \leq u\leq  \frac 32.$
We first see that by \eqref{integral between a and sigma}, the part of the integral with $\frac 43 \leq u \leq \frac 32 $ is
\be \ll\ Q^{-\frac 23} + Q^{\sigma-2}\log Q. \ee

We turn to the part of the integral with $1+\kappa \leq u\leq  \frac 43.$ We have by Theorem \ref{thm fiorilli} (setting $x:=Q^u$ and $M:=Q^{u-1}$) that it is
\begin{align}\label{eq:1to43rdsetakappa} &=  2\int_{1+\kappa}^{\frac 43} \left(\frac {f(u)}2 -\frac {f'(u)}{\log Q}\right) Q^{-u/2}\Big( -\frac 12 \log (Q^{u-1})- \frac{C_6}2  \notag
\\ &\hspace{6cm}+ O_{\epsilon}\left(Q^{\frac{1-u}{2} (1-\epsilon)}+Q^{\frac 32 u-2}(\log Q^u)^2\right) \Big)du \nonumber\\
&=    -\int_{1+\kappa}^{\frac 43}(  (u-1)\log Q+C_6  ) Q^{-u/2} \left(\frac {f(u)}2 -\frac {f'(u)}{\log Q}\right) du + O_{\epsilon}\left(\frac{Q^{-\frac 12-\kappa(1-\epsilon)}}{\log Q}+Q^{-\frac 23}\log Q\right),
\end{align}
hence \eqref{equation 3/2 3} holds.

\end{proof}

\begin{proof}[Proof of Theorem \ref{first thm}] We now turn to \eqref{equation 3/2 2}, with $f$ supported in $(-\frac 32, \frac 32)$.
Set $\kappa:=\frac{A\log\log Q}{\log Q}$ with $A\geq 1$ a constant. As the big-Oh constant in \eqref{eq:1to43rdsetakappa} is independent of $\kappa$, we may use \eqref{eq:1to43rdsetakappa} to estimate the contribution to \eqref{integral to understand} from $u \in [1+\kappa, \frac 43]$.  This part of the integral contributes \begin{align} -\int_{1+\kappa}^{4/3 }(  (u-1)\log Q+C_6  ) Q^{-u/2} \left(\frac {f(u)}2 -\frac {f'(u)}{\log Q}\right) du &+ O_{\epsilon}\left(\frac {Q^{-1/2}}{(\log Q)^{A(1-\epsilon)+1}}\right)\notag \\
&\ll \frac {Q^{-1/2}}{(\log Q)^{A/2}}.
\label{part of the integral bigger than 1+kappa}
\end{align}
The part of the integral with $\frac 43 \leq u \leq \frac 32$ was already shown to be $\ll Q^{-\frac 23}+Q^{\sigma-2}\log Q$, and hence is absorbed into the error term since $\sigma< 3/2$.

We now come to the heart of the argument, the part of the integral where $1\leq u \leq 1+\kappa$. Since $f\in C^2(\mathbb R)$, we have that in our range of $u$, $g(u):= \frac {f(u)}2 -\frac {f'(u)}{\log Q}$ satisfies
\begin{eqnarray} g(u) & \ =\ & \frac {f(1)}2 +\frac {f'(1)}2 (u-1) + O((u-1)^2) - \frac {f'(1)}{\log Q} + O\left(\frac{u-1}{\log Q}\right) \nonumber\\
&=& P(u-1) +O\left( \frac {(\log \log Q)^2}{(\log Q)^2} \right),
\end{eqnarray}
where $P(u) := \frac {f(1)}2- \frac {f'(1)}{\log Q} +\frac {f'(1)}2 u$.
At this point, if $f$ were $C^{K}(\mathbb R)$, we could take its Taylor expansion and get an error of $O_{\epsilon,A}\left( \frac {(\log \log Q)^{K}}{(\log Q)^{K}}\right)$.

We cannot apply Theorem \ref{thm fiorilli} directly since the error term is not got enough for moderate values of $M$. Instead, we argue as in the proof of Proposition 6.1 of \cite{fiorilli}. Slightly modifying the proof and using GRH, we get that
\begin{eqnarray}\label{expansion with r}& & \sum_{Q/2 < q\leq Q} \left( \psi(x;q,a)-\frac{\psi(x)}{\phi(q)} \right) \nonumber\\ & & =\ x\Bigg(-C_1 -\sum_{\substack{r< \frac {x-1}{Q}}}\frac 1{\phi(r)} \left(1-\frac r{x/{Q}}\right)
+\sum_{\substack{r< \frac {x-1}{Q/2}}}\frac 1{\phi(r)} \left(1-\frac r{2x/Q}\right)  \Bigg)
+O_{\epsilon}\left(\frac{x^{3/2+\frac{\epsilon}2}}{Q} \right),  \nonumber\\
 \end{eqnarray}
with
\begin{equation}
\label{equation definition C_1}
C_1 \ :=\ \frac{ \zeta(2)\zeta(3)}{\zeta(6)} \log 2.
\end{equation}
(We used that $\sum_{Q/2 < q\leq Q} = \sum_{Q/2 < q\leq x} -\sum_{Q < q\leq x}$, as in the proof of Theorem 4.1* of \cite{fiorilli2}.)
The contribution of the error term in \eqref{expansion with r} to the part of the integral in \eqref{integral to understand} with $1\leq u\leq 1+\kappa$ is (remember $\kappa \log Q = A\log\log Q$)
\be \ll\ \int_1^{1+\kappa}\frac{1}{Q/2} \frac{Q^{\frac{3u}2+\frac{\epsilon u}2} / Q}{Q^{u/2}}\ du\ \ll_{\epsilon}\ Q^{-1+ \epsilon}. \ee

Therefore, all that remains to complete the proof of Theorem \ref{first thm} it to estimate the contribution to \eqref{integral to understand} from $u \in [1, 1+\kappa]$. Using Lemma 5.9 of \cite{fiorilli} to bound the error in replacing $g(u)$ with $P(u-1)$, we find
\begin{align}  &2\int_1^{1+\kappa}g(u) \frac{Q^{\frac{u}{2}}}{Q/2}\Bigg(-C_1-\sum_{\substack{r< \frac {Q^u-1}Q}}\frac 1{\phi(r)} \left(1-\frac r{Q^{u-1}}\right) \notag \\
&\hspace{6cm}+\sum_{\substack{r< \frac {Q^u-1}{Q/2}}}\frac 1{\phi(r)} \left(1-\frac r{2Q^{u-1}}\right)\Bigg)du \notag\\
 &=\ 4\int_1^{1+\kappa}P(u-1) Q^{\frac u2-1}\Bigg(-C_1-\sum_{\substack{r\leq \frac {Q^u-1}Q}}\frac 1{\phi(r)} \left(1-\frac r{Q^{u-1}}\right)\notag\\
&\hspace{2cm}+\sum_{\substack{r\leq \frac {Q^u-1}{Q/2}}}\frac 1{\phi(r)} \left(1-\frac r{2Q^{u-1}}\right)\Bigg)du +O\left( \frac{Q^{-1/2} (\log\log Q)^2}{(\log Q)^{3}}\right); \label{simple integral}
\end{align}
we changed $r<\cdots$ to $r\leq \cdots$ in the sums above, which gives a negligible error term.

Setting $R:= Q^{\kappa} -\frac 1Q$, we compute that
\begin{align}
 &\int_1^{1+\kappa} P(u-1) Q^{\frac u2-1}\sum_{\substack{r\leq \frac {Q^u-1}Q}}\frac 1{\phi(r)} \left(1-\frac r{Q^{u-1}}\right)du \notag \\
&= \ \frac 1Q\sum_{\substack{r\leq Q^{\kappa} -\frac 1Q }}\frac 1{\phi(r)} \int_{1+\frac{\log\left( r+Q^{-1}\right)}{\log Q}}^{1+\kappa}P(u-1)\left(Q^{u/2}-\frac r{Q^{\frac u2-1}}\right)du \notag \\
&= \ \frac 1Q\sum_{\substack{r\leq R }}\frac 1{\phi(r)} \int_{1+\kappa \frac{\log r }{\log R}}^{1+\kappa}P(u-1)\left(Q^{u/2}-\frac r{Q^{\frac u2-1}}\right)du + O_{\epsilon}(Q^{-\frac 32+\epsilon}) \notag,
\end{align}
 the error term coming from the fact that we replaced $\log(r+Q^{-1})$ by $\log r$. Performing two changes of variables, we obtain that this is

\begin{align}
&= \ Q^{-1/2}\sum_{\substack{r\leq R }}\frac 1{\phi(r)} \int_{\kappa \frac{\log r }{\log R}}^{\kappa}P(u)\left(Q^{u/2}-\frac r{Q^{u/2}}\right)du + O_{\epsilon}(Q^{-\frac 32+\epsilon}) \notag  \\
&= \ Q^{-1/2}\sum_{\substack{r\leq R }}\frac 1{\phi(r)} \int_{\frac{\log r }{\log R}}^{1} \kappa P(\kappa v)\left(R^{\frac v2}-\frac r{R^{\frac v2}}\right)dv + O_{\epsilon}(Q^{-\frac 32+\epsilon}). \label{before applying lemma totient}
\end{align}
Let \be F_1 \ := \ -\ 4\zeta\left( \frac 12\right) \prod_p \left( 1+\frac 1{(p-1)p^{1/2}}\right) \text{ and } F_2 \ := \ F_1\left( \frac{\zeta'}{\zeta} \left(\frac 12\right) -\sum_p \frac{\log p}{(p-1)p^{\frac 12}+1}\right).\ee
By Lemma \ref{sum reciprocal totient lemma}, we find that \eqref{before applying lemma totient} equals
\begin{align} &\ \ \frac{\kappa}{Q^{1/2}} \Bigg(E_1 \log R  \int_{-\infty}^1 R^{u/2}vP(\kappa v)dv + E_2 \int_{-\infty}^1 R^{\frac v2}P(\kappa v)dv \notag  \\
 &\hspace{2cm} F_1\frac{P(\kappa)-\kappa P'(\kappa)}{\log R}+F_2\frac{-\kappa P'(\kappa)}{(\log R)^2} +O\left( R^{-1/2} \right)\Bigg) \notag \\
&=\ Q^{-1/2} \Bigg(E_1 \log Q  \int_{-\infty}^{\kappa} Q^{u/2}uP(u)du + E_2 \int_{-\infty}^{\kappa} Q^{u/2}P(u)du \notag \\
&\hspace{2cm}+ F_1\frac{\frac{f(1)}2-\frac{f'(1)}{\log Q}}{\log Q}-F_2\frac{f'(1)}{2(\log Q)^2} + O\left( R^{-1/2} \right)\Bigg).
\label{first thing to combine}
\end{align}

We obtain in an analogous way with $R:=2 Q^{\kappa}- \frac 2{Q}$ that
\begin{multline}
 \int_1^{1+\kappa} P(u-1) Q^{\frac u2-1}\sum_{\substack{r\leq 2\frac {Q^u-1}{Q}}}\frac 1{\phi(r)} \left(1-\frac r{2Q^{u-1}}\right)du \\
=\ Q^{-1/2}\sum_{\substack{r\leq R }}\frac 1{\phi(r)} \int_{\frac{\log (r/2) }{\log (R/2)}}^{1} \kappa P(\kappa v)\left((R/2)^{\frac v2}-\frac r{2(R/2)^{\frac v2}}\right)dv + O_{\epsilon}(Q^{-\frac 32+\epsilon}),
\end{multline}
which by Lemma \ref{sum reciprocal totient lemma} is
\begin{align}  &=\ \frac{\kappa}{Q^{1/2}} \Bigg(E_1 \log (R/2)  \int_{-\infty}^1 (R/2)^{\frac v2}vP(\kappa v)dv \notag \\
&\ \ \ \ \ \ +\ (E_2+E_1\log 2) \int_{-\infty}^1 (R/2)^{\frac v2}P(\kappa v)dv  \notag
+ \sum_{j=1}^n \frac{F_j^{(2)}}{(\log (R/2))^j}  + O(R^{-1/2})\Bigg) \\
&=\ Q^{-1/2} \Bigg( E_1 \log Q  \int_{-\infty}^{\kappa} Q^{u/2}uP(u)du
+ (E_2 +E_1\log 2)\int_{-\infty}^{\kappa} Q^{u/2}P(u)dv \notag \\
&\hspace{3cm}+ \frac{F_1}{\sqrt 2}\frac{\frac{f(1)}2-\frac{f'(1)}{\log Q}}{\log Q}-\frac{F_2+F_1\log 2}{\sqrt 2}\frac{f'(1)}{2(\log Q)^2}  + O(R^{-1/2})\Bigg).
\label{second thing to combine}
\end{align}

We now substitute \eqref{first thing to combine} and \eqref{second thing to combine} in \eqref{simple integral}, to get that \eqref{simple integral} is (notice the remarkable cancellations)

\begin{align} &=  -4C_1 \int_1^{1+\kappa} P(u-1)Q^{\frac u2-1}du +4E_1\log 2 Q^{-\frac 12}\int_{-\infty}^{\kappa} Q^{\frac u2} P(u)du    \notag \\
&\hspace{1cm}+4Q^{-\frac 12}\left(- F_1\frac{\frac{f(1)}2-\frac{f'(1)}{\log Q}}{\log Q}+F_2\frac{f'(1)}{2(\log Q)^2} + \frac{F_1}{\sqrt 2}\frac{\frac{f(1)}2-\frac{f'(1)}{\log Q}}{\log Q}-\frac{F_2+F_1\log 2}{\sqrt 2}\frac{f'(1)}{2(\log Q)^2}\right) \notag \\
&\hspace{4cm}+O\left( Q^{-\frac 12}\frac {(\log\log Q)^2}{(\log Q)^3} + \frac{Q^{-\frac 12}}{(\log Q)^{A/2}} \right),\notag
\end{align}
which by \eqref{equation definition E_i} and \eqref{equation definition C_1} is

\begin{align}
&= 4 \log 2 \frac{\zeta(2)\zeta(3)}{\zeta(6)} \int_{-\infty}^{1} P(u-1)Q^{\frac u2-1}du
 \notag \\
&\hspace{1cm}+(2-\sqrt 2) Q^{-\frac 12}\left(- F_1 \frac{f(1)}{\log Q}+ \bigg(F_2-\frac{\sqrt 2+4}3 F_1\bigg) \frac{f'(1)}{(\log Q)^2} \right) \notag \\
&\hspace{4cm}+O\left( Q^{-\frac 12}\frac {(\log\log Q)^2}{(\log Q)^3} + \frac{Q^{-\frac 12}}{(\log Q)^{A/2}} \right).
\end{align}
But, yet another cancellation is coming: we have that
\be \int_{-\infty}^{1} P(u-1)Q^{\frac u2-1}du\ =\ \int_{-\infty}^{1} g(u)Q^{\frac u2-1}du+O\left( Q^{-\frac 12}\frac {(\log\log Q)^2}{(\log Q)^3}\right),\ee
and so by \eqref{part of the integral u<1} this term cancels (up to the error term $O(Q^{-1})$) with the part of the integral of $T_4(Q)$ with $u\leq 1$ (which is coming from a totally different part of the problem, where there are no primes in arithmetic progressions involved)!

Combining all the terms,
\begin{eqnarray} \frac 1{Q/2} \sum_{Q/2<q\leq Q}T_4(Q) & = & (2-\sqrt 2) Q^{-1/2}\left(- F_1 \frac{f(1)}{\log Q}+ \bigg(F_2-\frac{\sqrt 2+4}3 F_1\bigg) \frac{f'(1)}{(\log Q)^2} \right)\nonumber\\  & & \ \ \ \ \ +\ O\left(\frac {Q^{-1/2}}{(\log Q)^{A/2}}+Q^{-1/2}\frac {(\log\log Q)^2}{(\log Q)^3}\right).
\end{eqnarray}

The proof is completed by taking $A=6$.

\end{proof}

\section{Results under De-averaging Hypothesis (Theorem \ref{thm 4-2eta})}\label{sec:proofthm42eta}

In this section we assume the de-averaging hypothesis (Hypothesis \ref{de-averaging hypothesis}), which relates the variance in the distribution of primes congruent to 1 to the average variance over all residue classes. Explicitly, we assume \eqref{eq:deavehypotheta} holds for some $\delta \in (0,1]$, and show how this allows us to compute the main term in the averaged 1-level density, $D_{1;Q/2,Q}(\widehat{f})$, for test functions $f$ supported in $[-4+2\delta, 4-2\delta]$. (Remember that this hypothesis is trivially true for $\delta = 1$, and expected to hold for any $\delta>0$.)

\begin{proof}[Proof of Theorem \ref{thm 4-2eta}] Starting from \eqref{prime sum term}, we have that
\be T_4(q) \ = \ 2 \int_0^{\infty} \left(\frac {f(u)}2 -\frac {f'(u)}{\log Q}\right)\frac{\psi(Q^u;q,1)-\frac{\psi(Q^u)}{\phi(q)}}{Q^{u/2}}\ du.  \ee Feeding this into Proposition \ref{proposition averaged explicit formula}, we are left with determining
\begin{equation}\label{crosspoint}
\frac1{Q/2}\sum_{Q/2<q\leq Q} T_4(q) \ = \ \frac1{Q/2}\int_0^{\sigma} \left(\frac {f(u)}2 -\frac {f'(u)}{\log Q}\right)Q^{-u/2}\sum_{Q/2<q\leq Q} \left(\psi(Q^u;q,1)-\frac{\psi(Q^u)}{\phi(q)} \right) du.
\end{equation}

We have already seen in the proof of Theorem \ref{thm unconditional} that the part of the integral in \eqref{crosspoint} with $0\leq u\leq 1$ is
$O(Q^{-1/2})$. For the part where $u\geq 1$, the Cauchy-Schwarz inequality shows that its contribution to the integral in \eqref{crosspoint} is
\be \ll\ \frac1{Q/2}\int_1^{\sigma} Q^{-u/2} \left| \sum_{Q/2<q\leq Q} \left(\psi(Q^u;q,1)-\frac{\psi(Q^u)}{\phi(q)} \right)^2 \right|^{1/2} \cdot \left| \sum_{Q/2<q\leq Q} 1^2 \right|^{1/2}  du. \ee
Now, by Hypothesis \ref{de-averaging hypothesis}, this is
\begin{equation} \ll\ \frac1{Q/2}\int_1^{\sigma} Q^{-u/2} \cdot Q^{\frac{\delta-1}2}\Bigg(\sum_{Q/2<q\leq Q} \sum_{\substack{1\leq a\leq q \\ (a,q)=1}}\left(\psi(Q^u;q,a)-\frac{\psi(Q^u)}{\phi(q)} \right)^2 \Bigg)^{1/2} \cdot Q^{1/2}\ du.
\label{equation where the variance appears}
\end{equation}
We now use a result of Goldston and Vaughan \cite{GV}, which states that under GRH we have for $1\leq Q\leq x$ that
\begin{eqnarray} & & \sum_{q\leq Q} \sum_{\substack{1\leq a\leq q \\ (a,q)=1}}\left(\psi(x;q,a)-\frac{\psi(x)}{\phi(q)} \right)^2  \nonumber\\ & & \ \ \ \ \ \ =\ Qx\log Q-cxQ + O_{\epsilon}\left( Q^2(x/Q)^{\frac 14+\epsilon}+x^{3/2} (\log 2x)^{5/2} (\log\log 3x)^2\right),
\label{equation GV}
\end{eqnarray}
where $c:=\gamma+\log 2\pi +1 +\sum_p \frac{\log p}{p(p-1)}$.

We now split the range of integration into the two subintervals $1\leq u \leq 2$ and $2\leq u\leq \sigma$. In the first range, we have that for $\epsilon>0$ small enough, $u+1\geq \max(7/4+u/4+\epsilon(u-1),3u/2)$, so \eqref{equation GV} implies that
\begin{equation}\sum_{q\leq Q} \sum_{\substack{1\leq a\leq q \\ (a,q)=1}}\left(\psi(x;q,a)-\frac{\psi(x)}{\phi(q)} \right)^2\ \ll\ Qx(\log x)^{3}
\label{bound for Q>x^1/2}
\end{equation}
(which, up to $x^{\epsilon}$, follows from Hooley's original result \cite{hooley}),
so we get that the part of \eqref{equation where the variance appears} with $1\leq u\leq 2$ is
\be \ll\ Q^{\frac{\delta}2-1} \int_1^{2} Q^{-u/2} Q^{\frac{u+1}{2}} (\log Q)^{3/2}\ du\ \ll\ Q^{\frac{\delta-1}2} (\log Q)^{3/2},\ee which is $o(1)$ if $\delta < 1$.

We now examine the second interval, that is $2\leq u\leq \sigma$. In this range, \eqref{equation GV} becomes
\begin{equation}\sum_{q\leq Q} \sum_{\substack{1\leq a\leq q \\ (a,q)=1}}\left(\psi(x;q,a)-\frac{\psi(x)}{\phi(q)} \right)^2\ \ll\ x^{3/2} (\log x)^{5/2}(\log\log x)^2
\label{bound for Q<x^1/2}
\end{equation}
(which, up to a factor of $x^{\epsilon}$, follows from Hooley's original result \cite{hooley}).
We thus get that the part of \eqref{equation where the variance appears} with $2\leq u\leq \sigma$ is
\be \ll\ \frac{Q^{\frac{\delta}2}}{Q/2} \int_2^{\sigma} Q^{-u/2} Q^{3u/4} (u\log Q)^{5/4} \log\log (Q^u) du\ \ll\ Q^{\frac{\sigma+2\delta}{4}-1} (\log Q)^{1/4} \log\log Q.\ee
If $\sigma<4-2\delta$ then the above is $o(1)$, completing the proof.
\end{proof}

\section{Results under Montgomery's Hypothesis (Theorem \ref{strongest thm})}\label{sec:montstrong}

We continue our investigations beyond the GRH, and assume a smoothed version of Montgomery's hypothesis, Hypothesis \ref{montgomery weakest}. Interestingly, this assumption allows us to compute the main term of the 1-level density, $D_{1;q}(\widehat{f})$, for test functions of arbitrarily large (but finite) support. While similar results have been previously observed \cite{MilSar}, we include a proof both for completeness and because these observations are not in the literature.

\begin{proof}[Proof of Theorem \ref{strongest thm}]
As we are fixing the modulus, we take $Q:=q$. By the explicit formula from Proposition \ref{proposition explicit formula}, we have
\begin{eqnarray}
D_{1;q}(\widehat{f}) & \ = \ & \frac{f(0)}{\log q} \left(  \log q -\log(8\pi e^{\gamma})-\sum_{p\mid q}\frac{\log p}{p-1}\right)
+\int_0^{\infty}\frac{f(0)-f(t)}{q^{t/2}-q^{-t/2}} dt \nonumber\\ && \ \ -\ \frac{2}{\log q} \left(\sum_{n\equiv 1 \bmod q}-\frac 1{\phi(q)}\sum_{n} \right)\frac{\Lambda(n)}{n^{1/2}} f\left(
\frac{\log n}{\log q}\right)+O\left(\frac{1}{\phi(q)}\right).
\end{eqnarray}

Let $\sigma:=\sup (\text{supp}f) < \infty$. We proved in \S\ref{section terms} that the only terms that are not $O(1/\log q)$ are the leading term $f(0)$ and possibly the prime sum, which we now study. We have
\be T_4(q) \ =
\ 2 \int_0^{\infty} \left(\frac {f(u)}2 -\frac {f'(u)}{\log q}\right)\frac{\psi(q^u;q,1)-\frac{\psi(q^u)}{\phi(q)}}{q^{u/2}}\ du.\ee
In the proof of Theorem \ref{thm unconditional} we determined that the part of the integral with $0\leq u\leq 1$ is $O(q^{-1/2})$. From the proof of Theorem \ref{thm millergoesal}, the part with $1\leq u\leq 2$ is $O(\frac{\log\log q}{\log q})$.\\


\begin{enumerate}

\item \emph{Proof of Theorem \ref{strongest thm}(1).} For the rest of the integral, we use Hypothesis \ref{montgomery weakest}. Note that $u\geq 2$, so $x=q^u\geq q^2$ with $u \le \sigma$, hence we can replace $o_{x\rightarrow \infty}$ by $o_{q\rightarrow \infty}$. An integration by parts gives that the rest of the integral is
\begin{eqnarray} &= \ & 0 - \left(\frac {f(2)}2 -\frac {f'(2)}{\log q}\right) \frac{\psi_2(q^2;q,1)-\frac{\psi_2(q^2)}{\phi(q)}}{q}  \nonumber\\
& & \ \ \ \ \ \ - \ 2 \int_0^{\infty} \left(\frac {3f(u)}4 -\frac {2f'(u)}{\log q} +\frac {f''(u)}{(\log q)^2}\right)\frac{\psi_2(q^u;q,1)-\frac{\psi_2(q^u)}{\phi(q)}}{q^{u/2}}\ du\nonumber\\
&= \ & \frac{o(q)}{q}+\int_{2}^{\sigma} \left(|f(u)|+|f'(u)|+|f''(u)|\right) \frac{o(q^{u/2})}{q^{u/2}} du \ = \ o(1),
\end{eqnarray}
proving the claim. Note that we are using the smoothed version of the prime sum.\\

\item \emph{Proof of Theorem \ref{strongest thm}(2).} We already know that the part of the integral with $0\leq u\leq 1$ is $\ll q^{-1/2}$. Taking $\epsilon:=\epsilon'/\sigma$ in Hypothesis \ref{montgomery weaker}, the rest of the integral is $O\left(\int_1^{\sigma} q^{\epsilon u-\theta }du\right)$, which is $O\left(q^{\epsilon'-\theta}\right)$ and thus negligible if we may take $\theta > 0$.

\end{enumerate}
\end{proof}

\begin{rek} Depending on our assumptions about the size of the error term in the distribution of primes in residue classes, we may allow $\sigma$ to grow with $Q$ at various explicit rates.
\end{rek}

\ \\

\end{document}